\documentclass[12pt, reqno]{amsart}
\usepackage{amsmath, amsthm, amscd, amsfonts, amssymb, graphicx, color, mathrsfs}
\usepackage[bookmarksnumbered, colorlinks, plainpages]{hyperref}
\usepackage{slashed}
\usepackage{mathtools}
\usepackage{enumitem}

\usepackage{lineno,hyperref}

 \newtheorem{thm}{Theorem}[section]
 \newtheorem{cor}[thm]{Corollary}
 \newtheorem{lem}[thm]{Lemma}
 \newtheorem{prop}[thm]{Proposition}
 \theoremstyle{definition}
 \newtheorem{defn}[thm]{Definition}
 \theoremstyle{remark}
 \newtheorem{rem}[thm]{Remark}

 \numberwithin{equation}{section}
\newcommand{\half}{\frac{1}{2}}

\newcommand{\ene}{\mathbb{N}}
\newcommand{\ar}{\mathbb{R}}
\newcommand{\ce}{\mathbb{C}}
\newcommand{\arn}{{\mathbb{R}}^n}
\newcommand{\Rn}{{\mathbb{R}}^n}

\newcommand{\ardn}{{\mathbb{R}}^{n}\times{\mathbb{R}}^{n}}

\newcommand{\bi}{\begin{itemize}}
\newcommand{\ei}{\end{itemize}}
\newcommand{\be}{\begin{enumerate}}
\newcommand{\ee}{\end{enumerate}}
\newcommand{\beq}{\begin{equation}}
\newcommand{\eq}{\end{equation}}
\newcommand{\Lh}{\mathcal{L}}

\newcommand{\jp}{\langle\xi\rangle}
\DeclareMathOperator{\Tr}{Tr}
\DeclareMathOperator{\os}{o}
\DeclareMathOperator{\obi}{O}
\newcommand{\Lie}{Lie}

\begin{document}

\title[ON A CLASS OF ANHARMONIC OSCILLATORS]{ON A CLASS OF ANHARMONIC OSCILLATORS }

\author[M. Chatzakou]{Marianna Chatzakou}

\address{Marianna Chatzakou: 
  \endgraf
  Department of Mathematics\\
Imperial College London\\
180 Queen's Gate, London SW7 2AZ\\
United Kingdom}

\email{m.chatzakou16@imperial.ac.uk}

\author[J. Delgado]{Julio Delgado}

\address{Julio Delgado: 
  \endgraf
 Universidad del Valle \\
Departamento de Matem\'aticas\\
Colombia}

\email{delgado.julio@correounivalle.edu.co}


\author[M. Ruzhansky]{Michael Ruzhansky}
\address{
  Michael Ruzhansky:
  \endgraf
  Department of Mathematics: Analysis, Logic and Discrete Mathematics
  \\
  Ghent University, Belgium
  \endgraf
  and
  \endgraf
  School of Mathematical Sciences \\
    Queen Mary University of London
  \\
  United Kingdom}
  \email{michael.ruzhansky@ugent.be}
  
\date{\today}
\begin{abstract}
 In this work we study a class of anharmonic oscillators within the framework of the Weyl-H\"ormander calculus. 
A prototype is an operator on $\arn$ of the form $(-\Delta)^{\ell}+|x|^{2k}$  for $k,\ell$  integers $\geq 1$. 
 We obtain spectral properties in terms of Schatten-von Neumann classes for their negative powers and derive from them estimates on the rate of growth for the eigenvalues of the anharmonic oscillator $(-\Delta)^{\ell}+|x|^{2k}$. In particular we give a simple proof for the main term of the spectral asymptotics of these operators.  We also study some examples of anharmonic oscillators arising from the analysis on Lie groups.\\
\end{abstract}


\maketitle
\tableofcontents
\section{Introduction}
In this paper we show some spectral properties for anharmonic oscillators on $\arn$. 
In the study of the Schr\"odinger equation $$i\partial_t\psi=-\Delta \psi +V(x)\psi$$ the analysis of energy levels $E_j$ is reduced to the corresponding  eigenvalue problem for the operator    $-\Delta  +V(x)$. Spectral properties for the quartic oscillator $V(x)=x^4$ on $\ar$,  have been studied by  Voros in \cite{vor:anosc} with methods also applicable to more general anharmonic oscillators $A_{2k}= -\frac{d^{2}}{dx^{2}}+x^{2k} $ on $\ar$. In this work, Voros used the zeta function $\zeta(s)$ of the anharmonic oscillator   $A_4=C(-\frac{d^{2}}{dx^{2}}+x^{4}),$ where $C$ is a suitable normalisation constant related to the Gamma function.  The zeta function of an operator $A$ is defined by  $\zeta(s)=\sum\limits_{j=0}^{\infty} \lambda_j^{-s}$, where $\lambda_j$ are the eigenvalues of $A$ arranged in the increasing order. It is important to point out that the zeta function $\zeta(s)$ of the harmonic oscillator  $A_2$ is related to the Riemann zeta function $\zeta_{R}(s)$ by
\[\zeta(s)=(1-2^{-s})\zeta_{R}(s).\]
The investigation of anharmonic oscillators is therefore not only  relevant in analysis and mathematical physics, but also in the number theory. We will also see how they are closely related to some Lie groups and provide some applications to the spectral theory of some differential operators. A special case will be the Heisenberg, Engel and Cartan groups which supply a family of interesting examples. 

Despite the intensive research on the quartic oscillator in the last 40 years, the exact solution for the corresponding eigenvalue problem is unknown (cf. \cite{osush:anh}). This fact is a further motivation for the research in approximative and qualitative methods around this problem. 

The more general anharmonic oscillators appeared in the literature in the form of quartic oscillators with potentials
$\lambda x^4+ \frac{x^2}{2}$ in the works of F. T. Hioe and E. W. Montroll (cf. \cite{hm:anh, hm:anh2}). There, the authors developed numerical calculations for the study of the corresponding energy levels. Subsequently, R. Balian, G. Parisi and A. Voros (cf. \cite{balpv:anh}) started the research on quartic oscillators from the point of view of Feynman path integrals.  In more recent works on anharmonic oscillators, S. Albeverio and S. Mazzucchi \cite{alb:anh} considered quartic Hamiltonians with time-dependent coefficients. The case of a fractional Laplacian and a quartic potential has been considered by S. Durugo and  J. L\"orinczi in \cite{dulo:anh}.  

A more general class of anharmonic oscillators  arises in the form $$-\frac{d^{2\ell}}{dx^{2\ell}}+x^{2k}+p(x),$$ where $p(x)$ is a polynomial of order $2k-1$ on $\ar$ and with $k, \ell$ integers $\geq 1$. The spectral asymptotics of such operators  have been analysed by   B. Helffer and D. Robert \cite{hr:sap3, hr:anosc2, hr:anosc}. In this paper we study the general case on $\arn$ where a prototype operator is of the form
 \beq (-\Delta)^{\ell}+|x|^{2k}\label{gahw1},\eq
where $k,\ell$ are integers $\geq 1$.  Spectral properties for this type of operators in the case $k=\ell$ have been considered in \cite{Helffer:b1} and  \cite{Robert:book1}. 

The general setting we consider here is as follows. Let $k$ be an integer $\geq 1$, and let $\mathcal{P}_{2k}$ be the set of real-valued polynomials on $\arn$, such that 
\[\liminf_{|x|\rightarrow\infty}\frac{p(x)}{|x|^{2k}}>0.\]
We consider operators of the form
\beq\label{qpop5} T=q(D)+p(x),\eq
where $q\in \mathcal{P}_{2\ell}, p\in\mathcal{P}_{2k}$, and $k,\ell$ are integers $\geq 1$. We note that if $p\in \mathcal{P}_{2k}$ then  $p(x)> -p_0$ for all $x\in\arn$ and for some $p_0>0$. Hence $p(x)+p_0>0.$ Analogously, by taking a $q_0>0$ such that $q(\xi)+q_0>0$ for all $\xi\in\arn$, we associate to the operator $T$ a H\"ormander metric $g$ given by
\beq g=g^{(p,q)}=\frac{dx^2}{(p_0+q_0+p(x)+q(\xi))^{\frac{1}{k}}}+\frac{d\xi ^2}{(p_0+q_0+p(x)+q(\xi))^{\frac{1}{\ell}}}\, .\label{anhm01}\eq 

We analyse then the operator $T$ within the framework of H\"ormander's $S(m,g)$ classes, the topic of Section \ref{SEC:anharmonic}. However, in Section \ref{SEC:anharmonic-notes} we give some simplifications from the presentational point of view, showing that the Weyl-H\"ormander classes corresponding to the metric \eqref{anhm01} for the operators \eqref{qpop5} can be described, independently of the lower order terms, by the following symbol classes:

\smallskip
{\em for $m\in\ar$, and $k, \ell$  integers $\geq 1$, the class   $\Sigma_{k,\ell}^m$ consists of all smooth functions $a\in C^{\infty}(\ardn)$ such that
 \beq\label{sigmacli}|\partial_{x}^{\beta}\partial_{\xi}^{\alpha}a(x,\xi)|\leq C_{\alpha\beta}
 (1+|x|^{k}+|\xi|^{\ell})^{m-\frac{|\beta|}{k}-\frac{|\alpha|}{\ell}}
 \eq 
 holds for all $x,\xi\in\Rn$,}
see Definition \ref{DEF:sigmas} and the subsequent statements.
Such a description also allows one to make use of these classes and their applications without any profound knowledge of the Weyl-H\"ormander theory. For example, for $k=l=1$, one recovers the well-known family of Shubin classes associated to the harmonic oscillator. We refer to Section \ref{SEC:anharmonic-notes} for main properties of these classes $\Sigma_{k,\ell}^m$.

\medskip
Regarding the spectral analysis of the anharmonic oscillators, we will obtain 
 the Schatten-von Neumann properties for the negative powers of such operators in the setting of $S(m,g)$ classes. The membership to Schatten-von Neumann classes and Weyl inequalities allow one to deduce the rate of decay of the eigenvalues for those negative powers. By looking at  the inverses of these negative powers one immediately recovers estimates for the rate of growth of eigenvalues of the original anharmonic oscillators. We point out that the spectral asymptotics for these operators has also been studied in \cite{hr:sap3, hr:anosc} in a rather complicated way.  
  Here we will give a simple proof for the main term of the spectral asymptotics applying a result by Buzano and Toft \cite{Buzano-Toft:Schatten-Weyl-JFA-2010} for our analysis.  The investigation 
of such properties within these  classes started with H\"ormander \cite{hor:assy}.  Other works on Schatten-von Neumann classes within the Weyl-H\"ormander calculus can be found in \cite{Toft:Schatten-AGAG-2006}, \cite{Toft:Schatten-modulation-2008}. See also \cite{dr:suffkernel}, \cite{dr13:schatten}, \cite{dr13a:nuclp} for several symbolic and kernel criteria on different types of domains.  For the spectral theory of non-commutative versions of the harmonic oscillator we refer the reader to the interesting work of Parmeggiani et al \cite{pab2:b}, \cite{par:1a}, \cite{par:1b}, \cite{par:1c}, \cite{par:1e} and \cite{par:1f}.\\

\medskip
The main results of this work give the order of the corresponding Schatten-von Neumann class for the negative powers of anharmonic oscillators.  The special case of the trace class is also distinguished. That is the contents of Theorem \ref{sch1mf} and Corollary \ref{cor.b.1}. From those we derive estimates for the rate of growth of anharmonic oscillators in \eqref{EQ:growthb}. We give a simple proof for the main term of the spectral asymptotics for our operators. We also provide examples of the above results for the anharmonic oscillators  arising in the study the Heisenberg, Engel and Cartan group in Section \ref{SEC:examples}.

\section[Preliminaries]{Preliminaries}
In this section we first briefly review some basic elements of the Weyl-H{\"o}rmander calculus.  For a comprehensive study
 on this important theory we refer the reader to  \cite{ho:apde2}, \cite{le:book}, \cite{b-l:qua}. Second, we  recall basic properties of Schatten-von Neumann classes. \\
 
We recall below the main quantizations that we are going to use:

\begin{defn} For $a=a(x,\xi)\in S'(\ardn)$ ($x\in \mathbb{R}^{n}$ and
$\xi\in \mathbb{R}^{n}$) and $t\in\ar$, we define the {\em t-quantization} as the operator $a_t(x,D):S(\mathbb{R}^{n})\rightarrow S(\mathbb{R}^{n})$ given by \[a_t(x,D)u(x)=(2\pi)^{-n}\int_{\arn}\int_{\arn} e^{i(x-y)\xi}a(tx+(1-t)y,\xi)u(y)dyd\xi.\]
\end{defn}

The case $t=1$ is known as the {\em Kohn-Nirenberg  quantization} and we also write $a(x,D)$ instead of $a_1(x,D)$. The case $t=\half$
 is know as the {\em Weyl quantization} and we also denote it by $a^w$. 
 
 Quantizations with nonlinear functions of $x$ and $y$ instead of  $tx+(1-t)y$ have been recently considered in \cite{er}. 

The Weyl quantization has fundamental relations with the symplectic structure of $\arn\times\arn=T^* \arn$; see Theorem \ref{thm.com.} on the composition of Weyl quantizations.  One of those already arises from the symbol of the composition.
\begin{defn} Let $a(x,\xi),b(x,\xi)\in S(\ardn)$. We define \[(a\#b)(X)=\pi^{-2n}\int_{\mathbb{R}^{2n}\times \mathbb{R}^{2n}}e^{-2i\sigma(X-Y_{1},X-Y_{2})}a(Y_{1})b(Y_{2})dY_{1}dY_{2},\]
where $\sigma(X,Y)=y\cdot\xi-x\cdot\eta$ for $X=(x,\xi)$ and $Y=(y,\eta)$.\end{defn}

The operation $\#$ becomes useful in order to describe the composition $a^w\circ b^{\omega}$, %
 indeed one has $a^w\circ b^{\omega}=(a\#b)^{\omega}$, see Theorem \ref{thm.com.}.\\

We shall now recall the definition of H{\"o}rmander metrics on the phase space.
\begin{defn}\label{defmet98} For $X \in \ardn$ let $g_{X}(\cdot)$ be a positive definite quadratic form on $\ardn$. We say that
$g_\cdot(\cdot)$ is a H\"ormander's metric if the following three conditions are satisfied:
\begin{enumerate}
\item {\bf Continuity or slowness}- There exist a constant $C>0$ such that 
\[g_{X}(X-Y)\leq C^{-1}\implies \left(\frac{g_X(T)}{g_Y(T)}\right)^{\pm 1}\leq 1,\]
for all $T\in\ardn$.
\item {\bf Uncertainty principle}- For $Y=(y,\eta)$ and
$Z=(z,\zeta)$, we define\\ $\sigma(Y,Z):=z \cdot \eta-y\cdot \zeta$,
and \[g_{X}^{\sigma}(T):=\sup_{W\neq 0}
\frac{\sigma(T,W)^{2}}{g_{X}(W)}.\] We say that $g$ satisfies the {\em uncertainty principle }
if \[\lambda_{g}(X)=\inf_{T\neq 0}
\left(\frac{g_{X}^{\sigma}(T)}{g_{X}(T)}\right)^{1/2}\geq 1 ,\] for
all $X, T\in\ardn$. The {\em   uncertainty parameter} or the {\em  Planck function} associated to $g$ is defined by
 $h_g(X)=(\lambda_{g}(X))^{-1}$.\\
\item {\bf Temperateness}- We say  that $g$ is temperate if there exist $\overline{C}>0$ and $J\in \mathbb{N}$ such that
\beq \left( \frac{g_{X}(T)}{g_{Y}(T)}\right)^{\pm1}\leq \overline{C}(1+g_{Y}^{\sigma}(X-Y))^{J},\label{tempc1x}\eq
for all $X, Y, T\in\ardn$.
\end{enumerate}
\end{defn}

\begin{rem}\label{rem.split.m} (i) For a {\em{split metric}} $g$, i.e. for a metric of the type
\[g_{X}(dx,d\xi)=\sum_{i=1}^{n}\left(\frac{dx_{i}^{2}}{a_{i}(X)}+\frac{d\xi_{i}^{2}}{b_{i}(X)}\right),\]
where $a_{i}(X)$ and $b_{i}(X)$ are positive functions, one can prove that 
\[g^{\sigma}_{X}(dx,d\xi)=\sum_{i=1}^{n}\left(b_{i}(X)dx_{i}^{2}+a_{i}(X)d\xi_{i}^{2}\right).\]
(ii) A special case of (i) is the one of {\em{symmetrically split metric}}, i.e. a metric of the type 
\[g_{X}(dx,d\xi)=\frac{dx^{2}}{a(X)}+\frac{d\xi^{2}}{b(X)}.\]

The metric \eqref{anhm01} is an example of a such type.\\
(iii) If $g$ is a split metric  one can prove the following formula for $\lambda_g$ from Definition \ref{defmet98},
\beq\lambda_{g}(X)=\min _j\sqrt{a_j(X)b_j(X)}\label{lambda1a}.\eq
(iv) In particular, if $g$ is symmetrically split like in (ii) then
\beq\lambda_{g}(X)=\sqrt{a(X)b(X)}\label{lambda1b}.\eq
\end{rem}
The classical weight $\langle\xi\rangle^m=(1+|\xi|)^m$ is generalised in the following way for a corresponding H{\"o}rmander metric.
\begin{defn} Let $M:\ardn\rightarrow (0,\infty)$ be a function. We say that 
$M$ is $g$-{{\em continuous}} if there exists $\tilde{C}>0$
such that
\[g_{X}(X-Y)\leq \frac{1}{\tilde{C}}\implies\left( \frac{M(X)}{M(Y)}\right)^{\pm1}\leq \tilde{C}.\]
\end{defn}
\begin{defn}  Let $M:\ardn\rightarrow (0,\infty)$ be a function. We say that
$M$ is  $g$-{\em temperate} if there exist $\tilde{C}>0$ and $N\in \mathbb{N}$
such that
\[\left( \frac{M(X)}{M(Y)}\right)^{\pm1}\leq \tilde{C}(1+g_{Y}^{\sigma}(X-Y))^{N}.\]
We will say that  $M$ is a $g$-{\em weight} if it is $g$-continuous and $g$-temperate.
\end{defn}
We are now ready to define $S(M,g)$ classes of symbols.  
\begin{defn} For a H{\"o}rmander metric $g$ and a $g$-weight $M$,  we denote by $S(M,g)$ the set of all smooth functions $a$ on $\ardn$ such that for any integer $k$ there exists $C_{k}>0$, such that for all
$X,T_{1},...,T_{k}\in \ardn$ we have 
\beq |a^{(k)}(X;T_{1},...,T_{k})|\leq C_{k}M(X)\prod_{i=1}^{k} g_{X}^{1/2}(T_{i}) .\label{inwhk}\eq
The notation $a^{(k)}$ stands for the $k^{th}$ derivative of $a$ and  $a^{(k)}(X;T_{1},...,T_{k})$ denotes the $k^{th}$  derivative of $a$ at $X$ in the directions $T_{1},...,T_{k}$. 
For $a\in S(M,g)$ we denote by $\parallel a\parallel_{k,S(M,g)}$ the minimum $C_{k}$ satisfying the above inequality. The class $S(M,g)$ becomes a Fr\'echet space endowed with the family of seminorms $\parallel \cdot\parallel_{k,S(M,g)}$.\end{defn}
One has the following composition theorem, see H\"{o}rmander \cite{ho:wc}.
\begin{thm}\label{thm.com.}
	Let $a \in S(M_1,g)\,,b \in S(M_2,g)$. Then,
	\[
	a^wb^w=(a\#b)^w\,,
	\]
	where for any given $N \in \mathbb{N}$
	\[
	(a\#b)(X)=\sum_{j=0}^{N-1} \frac{1}{j!}\left(\frac{i}{2}\sigma(D_X,D_Y)\right)^j a(X)b(Y)|_{X=Y}+r_N(X)\,,
	\]
	where $r_N \in S(M_1M_2 h_{g}^{N},g)$. In particular, for a given $N \in \mathbb{N}$,
	\[
	\sigma(D_X,D_Y)a(X)b(Y)|_{X=Y} \in S(M_1M_2 h_{g}^{N},g)\,.
	\]
\end{thm}
We note that by $\sigma(\cdot,\cdot)$ we have denoted the symplectic form on $T^{*}\arn$; that is for $X=(x,\xi)$ and $Y=(y,\eta)$,
\[
\sigma(D_X,D_Y):=D_yD_{\xi}-D_xD_{\eta}\,.
\]
\begin{rem}\label{se2a} 
Let $g$ be a H\"ormander metric. Instead of using the Planck function for the formulation
 of the statements we can equivalently employ the weight $\lambda_g$ defined by $\lambda_g=h_g^{-1}$. The function $\lambda_g$ is a $g$-weight for the metric $g$ (cf. \cite{ho:apde2}). Given a $g$-weight $M$, it is possible to construct an equivalent smooth weight $\widetilde{M}$ such that $\widetilde{M}\in S(M,g)$ (cf. \cite{ho:apde2}, \cite{le:book}). In particular, for $\lambda_g$ there exists an equivalent smooth weight $\widetilde{\lambda_g}$ such that $\widetilde{\lambda_g}\in S(\lambda_g,g)$. Hence, $\widetilde{\lambda_g}\in S(\widetilde{\lambda_g},g)$. A weight $M$ such that $M\in S(M,g)$ is called {\em regular}.   
Thus the weight $\lambda_g$ and consequently the Planck function $h_g$ can be assumed to be regular.

Let $0\leq \delta\leq\rho\leq 1$ and $\delta<1$. The metric $g^{\rho,\delta}$ is defined by
\[g_X^{\rho,\delta}(dx,d\xi)=\jp^{\delta}dx^{2}+\frac{d\xi^{2}}{\jp^{\rho}},\]
where $\jp:=(1+|\xi|^2)^{\half}$. It is well known that $g^{\rho,\delta}$ is a H\"ormander metric, and that with it one recovers the $S_{\rho,\delta}^m$ classes i.e., $S_{\rho,\delta}^m=S(\jp^m,g^{\rho,\delta})$. The uncertainty parameter  $\lambda_g$ for $g=g^{\rho,\delta}$ is given by 
 \[\lambda_g(X)=\langle\xi\rangle^{(\rho-\delta)}.\]
In the special case $\rho=1,\delta=0$, one has $\lambda_g(X)=\langle\xi\rangle$. The weight $\lambda_g$ can be seen as an extension of the basic one $\langle\xi\rangle$, for the $(\rho,\delta)$ classes. The symbols in $S(\lambda_g^{\mu},g)$ for $\mu\in\ar$ can be seen as the symbols of {\em order }  $\mu$ with respect to the metric $g$. In particular, $\lambda_g^{\mu}$ is a symbol of order $\mu$ with respect to $g$.
\end{rem}

We now recall the definition of Sobolev spaces adapted to the Weyl-H{\"o}r\-man\-der calculus. Here we adopt the Beals' definition for simplicity in the presentation of the basic theory. Comprehensive treatments on Sobolev spaces in this setting can be found in \cite{b-c:ef}, \cite{le:book}.

\begin{defn}\label{defi:sob} Let $g$ be a H{\"o}rmander metric and $M$ a $g$-weight. We will call 
{\em Sobolev space relative to }$M$ and it will be denoted by $H(M,g)$, the set of tempered distributions $u$ on $\arn$ such that
\beq 
a^wu\in L^2(\arn),\,\,\forall a\in S(M,g). 
\eq
\end{defn}

\begin{rem} \label{rem:1weights} We observe that the definition above requires a test over all the  symbols in $S(M,g)$. In contrast, we note that the classical Sobolev spaces $H^m=H(\langle\xi\rangle^m,g^{1,0})$ are defined by the condition on the tempered distribution $u$:
\[a(x,D)u\in L^2(\arn),\,\mbox{ with } a(x,\xi)=\langle\xi\rangle^m. \]  
\indent This means that, in the classical case the Sobolev space is defined by a fixed symbol. 
\end{rem}
The action of the Weyl quantization on the Sobolev spaces is determined by the following theorem (cf. \cite{le:book},  \cite{b-c:ef}).
\begin{thm}\label{tl2} Let $g$ be a H{\"o}rmander metric, $M$ and $M_{1}$ be two $g$-weights. For every $a\in S(M,g)$, we have $$a^w:H(M_{1},g)\rightarrow H(M_{1}/M,g).$$
\end{thm}
It is customary to identify $H(1,g)$ ($M=1$) with $L^{2}$ (cf.  \cite{le:book}, \cite{b-c:ef}):
\begin{thm} \label{l2}For a H\"ormander's metric $g$ we have $H(1,g)=L^{2}(\arn)$.\end{thm}
We shall now recall some basic properties of Schatten-von Neumann classes. Let $H$ be a complex separable Hilbert space endowed with an inner product denoted 
by $(\cdot,\cdot)$, and let $T:H\rightarrow H$ be a linear compact operator. If we denote by $T^*:H\rightarrow H$ the adjoint of  $T$, then the linear operator $(T^*T)^\half:H\rightarrow H$ is positive and compact. Let $(\psi_k)_k$ be an orthonormal basis for $H$ consisting of eigenvectors of $|T|=(T^*T)^\half$, and let $s_k(T)$ be the eigenvalue corresponding to the eigenvector 
$\psi_k$, $k=1,2,\dots$. The non-negative numbers $s_k(T)$, $k=1,2,\dots$, are called the singular values of $T:H\rightarrow H$. 
If 
$$
\sum_{k=1}^{\infty} s_k(T)<\infty ,
$$
then the linear operator $T:H\rightarrow H$ is said to be in the {\em trace class} $S_1$. It can be shown that  $S_1(H)$ is a Banach space in which the norm $\|\cdot\|_{S_1}$ is given by 
$$
\|T\|_{S_1}= \sum_{k=1}^{\infty} s_k(T),\,T\in S_1,
$$
multiplicities counted.
Let $T:H\rightarrow H$ be an operator in $S_1(H)$ and let  $(\phi_k)_k$ be any orthonormal basis for $H$. Then, the series $\sum\limits_{k=1}^{\infty} (T\phi_k,\phi_k)$   is absolutely convergent and the sum is independent of the choice of the orthonormal basis $(\phi_k)_k$. Thus, we can define the trace $\Tr(T)$ of any linear operator
$T:H\rightarrow H$ in $S_1$ by 
$$
\Tr(T):=\sum_{k=1}^{\infty}(T\phi_k,\phi_k),
$$
where $\{\phi_k: k=1,2,\dots\}$ is any orthonormal basis for $H$. If the singular values
are square-summable, $T$ is called a {\em Hilbert-Schmidt} operator. It is clear that every trace class operator is a Hilbert-Schmidt operator. More generally, if $0<p<\infty$ and the sequence of singular values is $p$-summable, then $T$ 
is said to belong to the Schatten-von Neumann class  ${S}_p(H)$, and it is well known that each ${S}_p(H)$ is an ideal in $\mathcal{L}(H)$. If $1\leq p <\infty$, a norm is associated to ${S}_p(H)$ by
 \[
 \|T\|_{S_p}=\left(\sum\limits_{k=1}^{\infty}(s_k(T))^p\right)^{\frac{1}{p}}.
 \] 
If $1\leq p<\infty$ 
 the class $S_p(H)$ becomes a Banach space endowed by the norm $\|T\|_{S_p}$. If $p=\infty$ we define $S_{\infty}(H)$ as the class of bounded linear operators on $H$, with 
$\|T\|_{S_\infty}:=\|T\|_{op}$, the operator norm.
In the case $0<p<1$  the quantity $\|T\|_{S_p}$ only defines
 a quasinorm, and $S_p(H)$ is also complete.

The Schatten-von Neumann classes are nested, with
\begin{equation}\label{EQ:Sch-nested}
{S}_p\subset {S}_q,\,\,\textrm{ if }\,\, 0<p<q\leq\infty.
\end{equation}

A  basic introduction to the study of the trace class is included in the book \cite{lax:fa} by Peter Lax. For the  basic theory of Schatten-von Neumann classes we refer the reader to \cite{gokr}, \cite{r-s:mp}, \cite{sim:trace}, \cite{sch:id}.


\section[Anharmonic oscillators]{Anharmonic oscillators}
\label{SEC:anharmonic}

In this section we specifically begin the study of our anharmonic oscillators. 
 We first introduce a special class of polynomials on $\arn$ which determine the operators that we will consider. For an integer $k\geq 1$ we define:
\beq
\mathcal{P}_{2k} =\left\{p:\arn\rightarrow\ar: p {\mbox{ is a polynomial of order }}\,\,2k {\mbox{ such that }}\,\liminf_{|x|\rightarrow\infty}\frac{p(x)}{|x|^{2k}}>0\right\}.
\eq
We now take $q=q(\xi)\in \mathcal{P}_{2\ell}$, $p=p(x)\in\mathcal{P}_{2k}$, where $k,\ell$ are integers $\geq 1$. We consider operators of the form
\beq\label{qpop} T=q(D)+p(x).\eq
We observe that since  $p\in \mathcal{P}_{2k}$ then  $p(x)> -p_0$ for all $x\in\arn$ and for some $p_0>0$. Hence $p(x)+p_0>0.$ Similarly,
there exists $q_0>0$ such that $q(\xi)+q_0>0$ for all $\xi\in\arn$. A prototype of the operators \eqref{qpop} is given in the form
\beq A=(-\Delta)^{\ell}+|x|^{2k}\label{gahw1b}\eq
where $k,\ell$ are integers $\geq 1$. 

We recall that spectral properties for operators of the form $-\frac{d^{2\ell}}{dx^{2\ell}}+x^{2k}+p_1(x)$, where $p_1$ is a suitable polynomial of order $2k-1$ have been studied by Helffer and Robert (cf. \cite{hr:anosc}, \cite{hr:anosc2}).\\

We associate to the operator $T=q(D)+p(x)$, the following metric 
\beq g=g^{(p,q)}=\frac{dx^2}{(p_0+q_0+p(x)+q(\xi))^{\frac{1}{k}}}+\frac{d\xi ^2}{(p_0+q_0+p(x)+q(\xi))^{\frac{1}{\ell}}}\, .\label{anhmetT}\eq 
It is clear that in the definition of  $g^{(p,q)}$ we can assume $p_0\geq 1$ obtaining an equivalent metric. \\

We note that if $k=\ell$ in $A$, the metric $g$ is equivalent 
to 
\beq \label{gkl}g=\frac{dx^2}{1+|x|^2+|\xi|^2}+\frac{d\xi ^2}{1+|x|^2+|\xi|^2},\eq 
which corresponds to the symplectic metric defining the  Shubin classes. However, the general case here is more delicate.\\

We now start by showing how the metric $g^{(p,q)}$ is constructed. In the following theorem, the membership $a\in S(L,G)$ should be understood in the sense that the inequality \eqref{inwhk} holds for a Riemannian metric $G$ on the phase-space and a strictly positive function $L$ on the phase-space, i.e. independent of the fact   whether or not $G$ is H\"ormander metric and $L$ just a $G$-weight.  These facts will be shown afterwards.

\begin{thm}\label{fgnyt}  Let $g=g^{(p,q)}$ be the metric defined by \eqref{anhmetT}. Then  
\[q(\xi)+p(x)\in S(p_0+q_0+q(\xi)+p(x),g^{(p,q)}).\]
\end{thm}
\begin{proof} In order to prove Theorem \ref{fgnyt} we first make some observations on the construction of the metric \eqref{anhmetT}:\\
\noindent {\bf Step 1.} The case $A=(-\Delta)^{\ell}+|x|^{2k}$. \\
We consider the prototype case of an operator of the form $A=(-\Delta)^{\ell}+|x|^{2k}$. 
 In this case the metric \eqref{anhmetT} is equivalent to 
\beq g^{(k,\ell)}=\frac{dx^2}{(1+|x|^{2k}+|\xi|^{2\ell})^{\frac{1}{k}}}+\frac{d\xi ^2}{(1+|x|^{2k}+|\xi|^{2\ell})^{\frac{1}{\ell}}}.\label{anhmet012}\eq 

We start by seeing  how the metric \eqref{anhmet012} arises from the operator $A$.
 The symbol of $A$ is $(2\pi)^{2\ell}|\xi|^{2\ell}+|x|^{2k}$, after  some rescaling and in order to simplify the calculations we consider instead the symbol  
\[\sigma(x,\xi)=|x|^{2k}+|\xi|^{2\ell}.\]

In order to obtain the metric \eqref{anhmet012} from the analysis of the derivatives of $\sigma(x,\xi)$ we will denote the metric we are searching for by $g$ and  we note that for $1\leq j\leq 2\ell$:

\[ |\partial_{\xi_i}^j (|x|^{2k}+|\xi|^{2\ell})|\leq (2\ell)^{j} |\xi|^{2\ell-j}.\]
On the other hand, we want to obtain an estimation of the form

\beq |\xi|^{2\ell-j}\leq C \frac{(1+|x|^{2k}+|\xi|^{2\ell})}{(1+|x|^{s}+|\xi|^{t})^j},\eq
where $s,t$ have to be found in order to get an inequality of the form \eqref{inwhk} for a suitable coefficient of a quadratic form in $d\xi ^2$. The values of $s,t$ will lead to \eqref{anhmet012}.\\

Equivalently, we should get for $0\leq j\leq 2\ell$:
\beq\label{eqwiu} |\xi|^{2\ell-j}(1+|x|^{js}+|\xi|^{jt})\leq C (1+|x|^{2k}+|\xi|^{2\ell}).\eq
The inequality  \eqref{eqwiu} lead us to two conditions (for large $|x|, |\xi|$):
\[(i)\,\, |\xi|^{2\ell-j}|\xi|^{jt}=|\xi|^{2\ell-(1-t)j}\leq C|\xi|^{2\ell}, \,\,(ii)\,\,|\xi|^{2\ell-j}|x|^{js}\leq C( |x|^{2k}+|\xi|^{2\ell}). \]
From (i) we obtain $t\leq 1$. In order to get (ii) we will find $1<p ,q<\infty$ where $\frac 1p+\frac 1q=1$ and such that
\beq\label{eqwiu32}|\xi|^{2\ell-j}|x|^{js}\leq C(|x|^{jsp}+|\xi|^{(2\ell-j) q})\leq C(|x|^{2k}+|\xi|^{2\ell}),\eq
from which we get the conditions 
\[(2\ell-j)q\leq 2\ell \,\,\mbox{ and }\,\, jsp\leq 2k .\]
 Therefore, for $j<2\ell$ we can take
\[q=\frac{2\ell}{2\ell-j}\geq 1,\] 
and observe that since $0\leq j\leq 2\ell-1 $ we have $ \frac{1}{2\ell}\leq\frac 1q \leq 1 .$ 
Now from $jsp\leq 2k$ and $\frac 1p\leq 1$ we obtain
\beq s\leq \frac{2k}{j}\frac 1p\leq \frac{2k}{j}, \,\,\mbox{ for all } j<2\ell .\label{jpsibb}\eq
On the other hand, $\frac 1j\geq \frac 1{2\ell}$. Hence, we have $s\leq \frac{2k}{2\ell}=\frac{k}{\ell}$.\\

We have omitted above the case $j=2\ell$, but its  analysis is simpler since $|\partial_{\xi_i}^{2\ell} (|x|^{2k}+|\xi|^{2\ell})|\leq (2\ell) ! $ and hence   \eqref{eqwiu} is still valid for $0\leq j\leq 2\ell$. In this way we get the coefficient of the metric for $d\xi^2$:
\[\frac{d\xi^2}{(1+|x|^{\frac{k}{\ell}}+|\xi|)^2}.\]

A similar analysis for  $|\partial_{x_i}^j (|x|^{2k}+|\xi|^{2\ell})|$, and exchanging the roles of $k$ and $\ell$ in the above estimations lead us to the coefficient of the metric for $dx^2$:
\[\frac{dx^2}{(1+|x|+|\xi|^{\frac{\ell}{k}})^2}.\]
Therefore, we have obtained the  metric 
\beq g(dx,d\xi)=\frac{dx^2}{(1+|x|+|\xi|^{\frac{\ell}{k}})^2}+\frac{d\xi^2}{(1+|x|^{\frac{k}{\ell}}+|\xi|)^2}\label{mety7} 
\eq
which completes the construction of the metric $g$ by observing that \eqref{mety7} is equivalent to \eqref{anhmet012}. It also proves that
\beq \label{mwfgt}|x|^{2k}+|\xi|^{2\ell}\in S(1+|x|^{2k}+|\xi|^{2\ell},g^{(k,\ell)}).\eq


In the search of $p$ for the analysis of \eqref{eqwiu} we have used the classical Young's inequality. 


Indeed by Young's inequality one has $uv\leq C_p(u^p+v^q),$ and in our case we have considered $u=|x|^{js}$ in order to get the estimate \eqref{eqwiu32}.\\

\noindent {\bf Step 2.} The case $\sigma_{ab}(x,D)=a(-\Delta)^{\ell}+b|x|^{2k}$ for $a,b>0$. \\

Here the metric \eqref{anhmetT} is of the form  
\beq\label{metabx} g^{(a,b)}=\frac{dx^2}{(1+b|x|^{2k}+a|\xi|^{2\ell})^{\frac{1}{k}}}+\frac{d\xi ^2}{(1+b|x|^{2k}+a|\xi|^{2\ell})^{\frac{1}{\ell}}}.\eq
In this case we should consider $\sigma_{ab}(x,\xi)=b|x|^{2k}+a|\xi|^{2\ell}$. We observe that $\sigma_{ab}(x,\xi)=|b^{\frac{1}{2k}}x|^{2k}+|a^{\frac{1}{2\ell}}\xi|^{2\ell}=\sigma(b^{\frac{1}{2k}}x, a^{\frac{1}{2\ell}}\xi )$. Now let $\alpha, \beta\in\ene_0^n$ , from \eqref{mwfgt} there exists a constant $C_{ab\alpha\beta}>0$ such that 
 \begin{align*} & |\partial_{x}^{\beta}\partial_{\xi}^{\alpha}\sigma(b^{\frac{1}{2k}}x, a^{\frac{1}{2\ell}}\xi))|  \\
 \leq & C_{ab\alpha\beta}(1+|b^{\frac{1}{2k}}x|^{2k}+|a^{\frac{1}{2\ell}}\xi|^{2\ell})\frac{1}{(1+|b^{\frac{1}{2k}}x|^{2k}+|a^{\frac{1}{2\ell}}\xi|^{2\ell})^{\frac{|\beta|}{2k}+\frac{|\alpha|}{2\ell}}}\\
\leq & C_{ab\alpha\beta}(1+b|x|^{2k}+a|\xi|^{2\ell})\frac{1}{(1+b|x|^{2k}+a|\xi|^{2\ell})^{\frac{|\beta|}{2k}+\frac{|\alpha|}{2\ell}}}.
\end{align*}
Therefore 
\beq \sigma_{ab}\in S(1+b|x|^{2k}+a|\xi|^{2\ell},g^{(a,b)})\label{sab3n}.\eq
The constants $C_{ab\alpha\beta}$ can be expressed in the form $C_{ab\alpha\beta}=C_{\alpha\beta}\cdot\max\{a,b,1\}$, where  
$C_{\alpha\beta}$ are the same structural constants as for the membership of  $\sigma$ in the symbol class $S(1+|x|^{2k}+|\xi|^{2\ell},g^{(k,\ell)})$. \\

\noindent {\bf Step 3.} The case $ T=q(D)+p(x)$, where  $q\in \mathcal{P}_{2\ell}$, $p\in\mathcal{P}_{2k}$, $k,\ell$  integers $\geq 1$. \\

In this case the symbol of $T$ is equivalent to  $T(x,\xi)=q(\xi)+p(x)$. From the assumptions on $q$ and $p$ there are $a,b>0$ and $q_0, p_0>0$ such that
\beq |p(x)|\leq b |x|^{2k}+p_0\, ,\,\,\,|q(\xi)|\leq a |\xi|^{2\ell}+q_0\label{dfx7}.\eq
Hence $|T(x,\xi)|\leq |q(\xi)|+|p(x)|\leq a |\xi|^{2\ell}+b |x|^{2k}+p_0+q_0 $. On the other hand, it is clear that from \eqref{sab3n} and comparing the  partial derivatives of $q(\xi)$ with the ones of $\sigma_{ab}$ we have
\begin{align*}|\partial_{x}^{\beta}\partial_{\xi}^{\alpha} T(x,\xi)|\leq &C_{ab\alpha\beta}|\partial_{x}^{\beta}\partial_{\xi}^{\alpha} (p_0+q_0+a|\xi|^{2\ell}+b|x|^{2k})|\\
\leq & C_{ab\alpha\beta} (p_0+q_0+b|x|^{2k}+a|\xi|^{2\ell})\frac{1}{(p_0+q_0+b|x|^{2k}+a|\xi|^{2\ell})^{\frac{|\beta|}{2k}+\frac{|\alpha|}{2\ell}}},
\end{align*}
where $C_{ab\alpha\beta}$ are as in Step 2.
 
Therefore,
\beq T(x,\xi)\in S(p_0+q_0+p(x)+q(\xi), g^{(p,q)}).\eq 
This completes the proof of Theorem \ref{fgnyt}.
\end{proof}

\begin{rem}\label{mtclz} We note that  by redefining the metric $g^{(a,b)}$ in \eqref{metabx} and letting
\beq\label{metabx1} g^{(a,b)}=\frac{\max\{a,b,1\}dx^2}{(1+b|x|^{2k}+a|\xi|^{2\ell})^{\frac{1}{k}}}+\frac{\max\{a,b,1\}d\xi ^2}{(1+b|x|^{2k}+a|\xi|^{2\ell})^{\frac{1}{\ell}}} \eq 
we only  left with the constants $C_{\alpha\beta}$ in the seminorms and therefore, they are only depending on $\alpha$ and $\beta$. A similar consequence can be deduced for the metric in Stage 3 to obtain structural constants independent of the coefficients in the polynomials, i.e., only dependent on $\alpha$ and $\beta$.
\end{rem}

We now turn to prove that  $g$ is indeed a H\"ormander metric. The lemma  below is useful to study the slowness property. In particular (ii) and (iii) help to reduce a proof of continuity. we refer the reader to (18.4.2), (18.4.2)' of \cite{ho:apde2} for the proof.
\begin{lem}\label{resl} Let $g$ be a Riemannian metric on the phase space. The following statements are equivalent:\\
 
\noindent (i) $g$ is continuous.\\
 
\noindent (ii) There exists a constant $C\geq 1$ such that
\[g_X(X-Y)\leq C^{-1}\,\,\mbox{ implies }\,\, g_Y\leq Cg_X.\]

\noindent (iii) There exists a constant $C\geq 1$ such that
\[g_X(Y)\leq C^{-1}\,\,\mbox{ implies }\,\, g_{X+Y}\leq Cg_X.\]
\end{lem} 
 
 We can now consider the metric  $g^{(p,q)}$ in detail.
\begin{thm}\label{contgk} The metric $g^{(p,q)}$ defined by \eqref{anhmetT} is a H\"ormander metric.
\end{thm} 
\begin{proof} By Remark \ref{rem.split.m} we compute
	\[\lambda_{g}(X)=(p_0+q_0+p(x)+q(\xi))^{\frac{k+\ell}{2k\ell}}\,.\]
	
	 As observed previously $p_0$ can be assumed $\geq 1$ obtaining an equivalent metric so that the uncertainty principle holds, i.e., we have $\lambda_{g}(X)\geq 1$ for all $X$.\\

For the proof of the continuity it will be enough to consider the case of the metric $g^{(k,\ell)}$, we will use Lemma \ref{resl}, (ii), 
 and  prove that 
\[\frac{|x-y|^2}{(1+|x|^{2k}+|\xi|^{2\ell})^{\frac{1}{k}}}+\frac{|\xi-\eta|^2}{(1+|x|^{2k}+|\xi|^{2\ell})^{\frac{1}{\ell}}}\leq C^{-1}\implies\]
\begin{multline*}
\frac{|t|^2}{(1+|y|^{2k}+|\eta|^{2\ell})^{\frac{1}{k}}}+\frac{|\tau|^2}{(1+|y|^{2k}+|\eta|^{2\ell})^{\frac{1}{\ell}}} \\
\leq C\left(\frac{|t|^2}{(1+|x|^{2k}+|\xi|^{2\ell})^{\frac{1}{k}}}+\frac{|\tau|^2}{(1+|x|^{2k}+|\xi|^{2\ell})^{\frac{1}{\ell}}}\right),
\end{multline*}
for all $t,\tau\in\ardn$.\\

We observe that the proof of this is reduced to prove that 
\[\frac{|x-y|^2}{(1+|x|^{2k}+|\xi|^{2\ell})^{\frac{1}{k}}}\, ,\,\, \frac{|\xi-\eta|^2}{(1+|x|^{2k}+|\xi|^{2\ell})^{\frac{1}{\ell}}}\leq C^{-1}\implies 
\frac{(1+|x|^{2k}+|\xi|^{2\ell})^{\frac{1}{k}}}{(1+|y|^{2k}+|\eta|^{2\ell})^{\frac{1}{k}}}\leq C .\]
Or even simpler
\beq\frac{|x-y|}{(1+|x|^{k}+|\xi|^{\ell})^{\frac{1}{k}}}\, ,\,\, \frac{|\xi-\eta|}{(1+|x|^{k}+|\xi|^{\ell})^{\frac{1}{\ell}}}\leq C^{-1}\implies 
\frac{1+|x|^{2k}+|\xi|^{2\ell}}{1+|y|^{2k}+|\eta|^{2\ell}}\leq C .\label{lhd1t}\eq

We will assume that the LHS of \eqref{lhd1t} holds for a constant $C>0$ to be chosen later on. Since $|x|-|y|\leq |x-y|$ and $|\xi|-|\eta|\leq |\xi-\eta|$, we have  

\beq |x|\leq  C^{-1}(1+|x|^{k}+|\xi|^{\ell})^{\frac{1}{k}}+|y|,\label{le23}
\eq

\beq |\xi|\leq  C^{-1}(1+|x|^{k}+|\xi|^{\ell})^{\frac{1}{\ell}}+|\eta|.\label{le24}
\eq
By taking powers $2k$ and $2\ell$ on \eqref{le23} and \eqref{le24} respectively, there exists a constant $C_1>1$  only dependent on $k$ and $\ell$ such that
\beq |x|^{2k}\leq  C^{-2k}C_1(1+|x|^{2k}+|\xi|^{2\ell})+C_1|y|^{2k},\label{le231b}
\eq

\beq |\xi|^{2\ell}\leq   C^{-2\ell}C_1(1+|x|^{2k}+|\xi|^{2\ell})+C_1|\eta|^{2\ell}.\label{le241x}
\eq
By adding \eqref{le231b} and \eqref{le241x}, and taking $k_0=\min\{k,\ell\}$ we obtain
\[1+|x|^{2k}+ |\xi|^{2\ell}\leq  2C^{-2k_0}C_1(1+|x|^{2k}+|\xi|^{2\ell})+1+C_1|y|^{2k}+C_1|\eta|^{2\ell}.\]
Hence and since  $C_1>1$ we have
\[(1+|x|^{2k}+ |\xi|^{2\ell})(1- 2C^{-2k_0}C_1)\leq  1+C_1|y|^{2k}+C_1|\eta|^{2\ell}\leq C_1(1+|y|^{2k}+|\eta|^{2\ell}).\]

By choosing $C$ such that $C^{2k_0}>2C_1$ we obtain the desired constant and
\[\frac{1+|x|^{2k}+|\xi|^{2\ell}}{1+|y|^{2k}+|\eta|^{2\ell}}\leq C_1(1- 2C^{-2k_0}C_1)^{-1}.\]
Therefore $g^{(k,\ell)}$ is continuous.\\ 

To see the general case, we note that in the previous proof of continuity for $g^{(k,\ell)}$ if we consider instead $g^{p,q}$, the analysis of the corresponding inequalities is reduced to the terms as for $g^{(k,\ell)}$. This completes the proof of the continuity.\\

We now prove the temperateness. We will only consider the case $g^{(k,\ell)}$. According to \eqref{tempc1x}, we need to prove that there exist $C>0$ and $N\in\ene$ such that
\begin{multline*}
\frac{|t|^2}{(1+|x|^{2k}+|\xi|^{2\ell})^{\frac{1}{k}}}+\frac{|\tau|^2}{(1+|x|^{2k}+|\xi|^{2\ell})^{\frac{1}{\ell}}} \\
\leq
 C\left(\frac{|t|^2}{(1+|y|^{2k}+|\eta|^{2\ell})^{\frac{1}{k}}}+\frac{|\tau|^2}{(1+|y|^{2k}+|\eta|^{2\ell})^{\frac{1}{\ell}}}\right)\times \\
 \times\left(1+(1+|x|^{2k}+|\xi|^{2\ell})^{\frac{1}{\ell}}|x-y|^2+(1+|x|^{2k}+|\xi|^{2\ell})^{\frac{1}{k}}|\xi-\eta|^2\right)^{N},
 \end{multline*}
 and 
 \begin{multline*}
 \frac{|t|^2}{(1+|y|^{2k}+|\eta|^{2\ell})^{\frac{1}{k}}}+\frac{|\tau|^2}{(1+|y|^{2k}+|\eta|^{2\ell})^{\frac{1}{\ell}}} \\ \leq
 C\left(\frac{|t|^2}{(1+|x|^{2k}+|\xi|^{2\ell})^{\frac{1}{k}}}+\frac{|\tau|^2}{(1+|x|^{2k}+|\xi|^{2\ell})^{\frac{1}{\ell}}}\right)\times \\
 \times\left(1+(1+|x|^{2k}+|\xi|^{2\ell})^{\frac{1}{\ell}}|x-y|^2+(1+|x|^{2k}+|\xi|^{2\ell})^{\frac{1}{k}}|\xi-\eta|^2\right)^{N},
  \end{multline*}
 for all $t, \tau\in\arn$.\\
 
We will only prove the first inequality since the other one has a similar argument.\\  
 
We now observe that we can reduce the proof of the first inequality to the following two inequalities:
 \begin{multline*}
\frac{(1+|y|^{2k}+|\eta|^{2\ell})^{\frac{1}{k}}}{(1+|x|^{2k}+|\xi|^{2\ell})^{\frac{1}{k}}} \\ \leq
 C\left(1+(1+|x|^{2k}+|\xi|^{2\ell})^{\frac{1}{\ell}}|x-y|^2+(1+|x|^{2k}+|\xi|^{2\ell})^{\frac{1}{k}}|\xi-\eta|^2\right)^{N},
   \end{multline*}
and
 \begin{multline*}
 \frac{(1+|y|^{2k}+|\eta|^{2\ell})^{\frac{1}{\ell}}}{(1+|x|^{2k}+|\xi|^{2\ell})^{\frac{1}{\ell}}} \\ \leq
 C\left(1+(1+|x|^{2k}+|\xi|^{2\ell})^{\frac{1}{\ell}}|x-y|^2+(1+|x|^{2k}+|\xi|^{2\ell})^{\frac{1}{k}}|\xi-\eta|^2\right)^{N}.
   \end{multline*}
Therefore, it will be enough to prove the first inequality, which again is reduced to
\beq\label{tempk4}\frac{1+|y|+|\eta|^{\frac{\ell}{k}}}{1+|x|+|\xi|^{\frac{\ell}{k}}}\leq
 C\left(1+(1+|x|^{\frac{ k}{\ell}}+|\xi|)|x-y|+(1+|x|+|\xi|^{\frac{\ell}{k}})|\xi-\eta|\right)^{N}.\eq
Now \eqref{tempk4} can be obtained from the inequalities below
\beq\label{tempk4d}\frac{1+|y|}{1+|x|+|\xi|^{\frac{\ell}{k}}}\leq \frac{1+|y|}{1+|x|}\leq
 C\left(1+(1+|x|^{\frac{ k}{\ell}})|x-y|\right)^{N},\eq
 \beq\label{tempk4a}\frac{1+|\eta|^{\frac{\ell}{k}}}{1+|x|+|\xi|^{\frac{\ell}{k}}}\leq\frac{1+|\eta|^{\frac{\ell}{k}}}{1+|\xi|^{\frac{\ell}{k}}}\leq
 C\left(1+(1+|\xi|^{\frac{\ell}{k}})|\xi-\eta|\right)^{N},\eq
which can be easily verified.
\end{proof}
In general one can assume that the uncertainty parameter $\lambda_g$ of a H\"ormander metric $g$ is a $g$-weight as noted in Remark \ref{se2a}. However here one can deduce it directly. Indeed, the proof of Theorem \ref{contgk} has some immediate consequences. In particular the validity of \eqref{lhd1t} implies that $m=p_0+q_0+p(x)+q(\xi)$ is  $g$-continuous. Similarly, we can obtain the temperateness of $p_0+q_0+p(x)+q(\xi)$ with respect to $g$ from the proof above for the temperateness of $g$. Therefore $p_0+q_0+p(x)+q(\xi)$ is a $g$-weight. Summarising we have:
\begin{prop} Let $g=g^{(p,q)}$ be the metric defined  by \eqref{anhmetT}. Then $p_0+q_0+p(x)+q(\xi)$ is a $g$-weight.
\end{prop}
\begin{rem}\label{repkl}
Now, from the construction of the metric $g=g^{(p,q)}$  we know that $$p(x)+q(\xi)\in S(p_0+q_0+p(x)+q(\xi),g).$$ On the other hand, 
since $\lambda_g(x,\xi)=(p_0+q_0+p(x)+q(\xi))^{\frac{k+\ell}{2k\ell}}$ we note that 
 $$p(x)+q(\xi)\in S(\lambda_g^{\frac{2k\ell}{k+\ell}},g),$$
 and therefore $p(x)+q(\xi)$ is a symbol of order $\frac{2k\ell}{k+\ell}$
  with respect to $g$. \\

In particular, for the quartic oscillator we have
\[|x|^4+|\xi|^2\in S(\lambda_g^{\frac 43},g^{(1,2)}),\]
i.e. the quartic oscillator is of order $\frac 43$ with respect to $g^{(1,2)}$. 
\end{rem}

 \begin{rem} Regarding the metric $g^{(k,\ell)}$ defined in  \eqref{anhmet012}, we have the following properties: \\
(i) Given $k_1,\ell_1, k_2, \ell_2\geq 1$, with 
either $k_1\neq k_2$ or $\ell_1\neq\ell_2$,  the metrics  $g^{(k_1,\ell_1)}$ and $g^{(k_2,\ell_2)}$ can not be compared, i.e., 
 none of the inequalities $g^{(k_1,\ell_1)}\leq g^{(k_2,\ell_2)}$,  $g^{(k_2,\ell_2)}\leq g^{(k_1,\ell_1)}$ hold. \\
(ii) On the other hand, we can compare any of the metrics $g^{(k,\ell)}$ with the metric $g^{1,0}$ in the sense of $(\rho, \delta)$ classes. Indeed, it is   clear  that 
\[g^{(k,\ell)}\leq g^{1,0},\]
 for all $k, \ell \geq 1$.
\end{rem}
As a consequence of the Theorem \ref{fgnyt} we have the following property for the membership of more general polynomials in classes determined by the metric $g$ defined in \eqref{anhmetT}:
\begin{cor} Let  $g=g^{(p,q)}$ be the metric defined  by \eqref{anhmetT}. 
If $p_1:\arn\rightarrow\ce$ is a polynomial of order $\leq 2k$ and $q_1:\arn\rightarrow \ce$ is a polynomial of order $\leq 2\ell$, then
\[p_1(x)+q_1(\xi)\in S(\lambda_{g}^{\frac{2k\ell}{k+\ell}},g).\]
\end{cor}

\section{Some notes on the classes $S(\lambda_g^{m},g)$}
\label{SEC:anharmonic-notes}

We now make some observations and deduce some consequences formulating a setting in a more intrinsic way. 

Let $g=g^{(p,q)}$ be defined by \eqref{anhmetT} with  $k, \ell$ integers $\geq 1$. Now that we know that $g$ is a H\"ormander metric, therefore there is a corresponding pseudo-differential calculus. We can define the class $S(\lambda_g^{m},g)$ for $m\in\ar$ in an equivalent way without referring  explicitly to the metric $g$. Indeed, we can define them in the following way:
\begin{defn}\label{DEF:sigmas} 
Let $m\in\ar$ and  let $k, \ell$ be integers $\geq 1$. If $a\in C^{\infty}(\ardn)$ we will say that $a\in\Sigma_{k,\ell}^m$ if 
 \beq\label{sigmacl}|\partial_{x}^{\beta}\partial_{\xi}^{\alpha}a(x,\xi)|\leq C_{\alpha\beta}(1+|x|^{k}+|\xi|^{\ell})^{m-\frac{|\beta|}{k}-\frac{|\alpha|}{\ell}}\,,\eq
 holds for all multi-indices $\alpha,\beta$, and all $x,\xi\in\mathbb R^n$.
\end{defn}
The definition above is related to the corresponding one for the metric $g=g^{(k,\ell)}$ defined by \eqref{anhmet012}, i.e., the one associated to $p(x)=|x|^{2k}, q(\xi)=|\xi|^{2\ell}$. Indeed we have:
\begin{prop}\label{sht57} Let $m\in\ar$ and let $g=g^{(k,\ell)}$ be defined by \eqref{anhmet012} with  $k, \ell$ integers $\geq 1$. Then 
\[\Sigma_{k,\ell}^m=S(\lambda_g^{m(\frac{k\ell}{k+\ell})},g).\]
\end{prop}

\begin{proof}
	Recall that for $g=g^{(k,\ell)}$, by Theorem \ref{contgk} we have  $$1+|x|^{k}+|\xi|^{\ell}\sim \lambda_{g}^{\frac{k\ell}{k+\ell}}.$$ Therefore, by \eqref{sigmacl}, we have $$S(\lambda_g^{m(\frac{k\ell}{k+\ell})},g)\subset \Sigma_{k,\ell}^{m}.$$ 
	On the other hand, if $a \in C^{\infty}(\arn \times \arn) \in S(\lambda_g^{m(\frac{k\ell}{k+\ell})},g)$, then by Definition \ref{inwhk} we have
	\[
	|\partial_{x}^{\beta}\partial_{\xi}^{\alpha}a(x,\xi)|\leq C_{\alpha\beta}(1+|x|^{2k}+|\xi|^{2\ell})^{\frac{m}{2}-\frac{|\beta|}{2k}-\frac{|\alpha|}{2\ell}}\,,
	\] 
	i.e., we get the reverse inclusion $S(\lambda_g^{m(\frac{k\ell}{k+\ell})},g) \subset \Sigma_{k,\ell}^{m}$, and this completes the proof.
\end{proof}
This observation will be helpful to undertake some investigations on anharmonic oscillators in a simplified way without an explicit reference to
the $S(M,g)$ setting.  

By associating to a symbol $a\in\Sigma_{k,\ell}^m$ a pseudodifferential operator $a(x,D)$ we  dispose of a pseudodifferential calculus on $\Sigma_{k,\ell}^m$ , inherited from the $S(M,g)$ calculus.\\

We also observe that actually Proposition \ref{sht57} is also valid if instead we use a metric $g^{(p,q)}$ with  polynomials 
 $p\in\mathcal{P}_{2k}, q\in \mathcal{P}_{2\ell}$. Moreover we have:
\begin{cor} Let $m\in\ar$ and let $k, \ell$ be integers $\geq 1$. For any $p\in\mathcal{P}_{2k}, q\in \mathcal{P}_{2\ell}$, the  
 classes $S(\lambda_{g^{(p,q)}}^{m\left(\frac{k\ell}{k+\ell} \right)}, g^{(p,q)})$ all coincide and are equal to $\Sigma_{k,\ell}^m$.
\end{cor}

We now formulate some few consequences for the classes $\Sigma_{k,\ell}^m$. In particular, for the composition formula we have:
\begin{thm}\label{thm.comp} Let $m_1, m_2\in\ar$ and  $k, \ell$ be integers $\geq 1$. If $a\in\Sigma_{k,\ell}^{m_1}, b\in\Sigma_{k,\ell}^{m_2}$, there exists $c\in \Sigma_{k,\ell}^{m_1+m_2}$ such that $a(x,D)\circ b(x,D)=c(x,D)$ and
\[c(x,\xi)\sim \sum\limits_{\alpha}(2\pi i)^{-|\alpha|}\partial_{\xi}^{\alpha}a(x,\xi)\partial_{x}^{\alpha}b(x,\xi),\]
i.e. for all $N\in\ene$
\[c(x,\xi)- \sum\limits_{|\alpha|<N}(2\pi i)^{-|\alpha|}\partial_{\xi}^{\alpha}a(x,\xi)\partial_{x}^{\alpha}b(x,\xi)\in \Sigma_{k,\ell}^{m_1+m_2-N\left(\frac{k+\ell}{k\ell}\right)} .\]
\end{thm}
\begin{rem} Recall that by the composition Theorem \ref{thm.com.} and Proposition \ref{sht57}, the reminder term in the above composition formula lies in the class
\[
S(\lambda_{g}^{(m_1+m_2)\left(\frac{k\ell}{k+\ell}\right)-N},g^{(k,\ell)})=\Sigma_{k,\ell}^{m_1+m_2-N\left(\frac{k+\ell}{k\ell}\right)}\,.
\]\end{rem}
The  $L^2$ boundedness of operators in the class $\Sigma_{k,\ell}^{0}$ is obtained from the corresponding one for the $S(1,g)$ class.
\begin{thm} Let $k, \ell$ be integers $\geq 1$. If $a\in\Sigma_{k,\ell}^{0}$, then $a(x,D)$ extends to a bounded operator $a(x,D):L^2(\arn)\rightarrow L^2(\arn)$. 
\end{thm}
For a more general consideration of boundedness results of operators with symbols in the classes $\Sigma_{k,\ell}^{m}$, let us introduce the below Sobolev spaces.
\begin{defn}
	Let $m\in\ar$ and $k,\ell$ integers $\geq 1$. We will denote by $H_{k,\ell}^m(\arn)$ the Sobolev space of order $m$ relative to $k, \ell$, that is the set of tempered distributions $u$ on $\arn$ such that
	\beq 
	(A+I)^{\frac{m}{2}}u\in L^2(\arn)\,,
	\eq 
	where $A=(-\Delta)^{\ell}+|x|^{2k}$.
\end{defn}

In the scale of those Sobolev spaces we have:
\begin{thm} Let $m\in\ar$ and $k, \ell$  integers $\geq 1$. If $a\in\Sigma_{k,\ell}^{m}$, then $a(x,D)$ extends to a bounded operator 
\beq a(x,D):H_{k,\ell}^s(\arn)\rightarrow H_{k,\ell}^{s-m}(\arn)\label{bxdsob}\eq
for all $s\in\ar$.
\end{thm}

Of course the above theorems on the boundedness also hold for $t$-quan\-ti\-zations $a_t(x,D)$ in \eqref{bxdsob}, and in particular for the Weyl quantization due to the corresponding general property in the $S(m,g)$ setting on the switching between those quantizations for split metrics.

In what follows we always denote $A=(-\Delta)^{\ell}+|x|^{2k}$.
The Sobolev spaces $H^{m}_{k,\ell}(\arn)$ satisfy the following properties:

\begin{thm}\label{sob.prop}
	\begin{enumerate}
		\item \label{itm: 1.sb}The space $H^{m}_{k,\ell}(\mathbb{R}^n)$ is a Hilbert space endowed with the sesquilinear form
		\[
		(g,h)_{H^m}=\left( (A+I)^{\frac{m}{2}}g,(A+I)^{\frac{m}{2}}h \right)_{L^2(\mathbb{R}^n)}\,.
		\]
		We have the inclusions 
		\[
		\mathcal{S}(\mathbb{R}^n) \subset H^{m_1}_{k,\ell}(\mathbb{R}^n) \subset H^{m_2}_{k,\ell}\subset \mathcal{S}^{'}(\mathbb{R}^n)\,,\quad m_1>m_2\,,
		\]
		while also that 
		\[
		L^2(\mathbb{R}^n)=H^{0}_{k,\ell}(\mathbb{R}^n)\,, \quad \text{and}\quad \mathcal{S}(\mathbb{R}^n)=\underset{m \in \mathbb{R}}{\bigcap}H^{m}_{k,\ell}(\mathbb{R}^n)\,.
		\]
		\item\label{itm:2.sb} The space $H^{-m}_{k,\ell}(\mathbb{R}^n)$ can be identified with the dual space of $H^{m}_{k,\ell}(\mathbb{R}^n)$. 
		\item \label{itm:3.sb} For any $m \in \mathbb{R}$, $H^{m}_{k,\ell}(\mathbb{R}^n)$ coincides with the Sobolev space $H^{(b)}_{m}(\mathbb{R}^n)$; that is the completion of the Schwartz space $\mathcal{S}(\mathbb{R}^n)$ in $\mathcal{S}^{'}(\mathbb{R}^n)$ for the norm 
		\begin{equation}\label{promise}
			\|h\|_{H_{m}^{(b)}}=\|(b^m)^{w}h\|_{L^2(\mathbb{R}^n)}\,,
		\end{equation}
		where $b(x,\xi)=1+|x|^{k}+|\xi|^{\ell}$. 
		\item \label{itm:4.sb} The Sobolev space of order $m$ relative to $k,l$, or in symbols the space $H^{m}_{k,\ell}$, coincides with the Sobolev space associated with the metric $g^{(k,\ell)}$; that is we have
		\begin{equation}\label{sobolev.iden}
			H^{m}_{k,\ell}(\mathbb{R}^n)=H \left( (1+|x|^{k}+|\xi|^{\ell})^{m},g^{(k,\ell)} \right)\,.
		\end{equation} 
		
		\item \label{itm: 6.sb} The complex interpolation between the spaces $H^{m_0}_{k,\ell}(\arn)$ and $H^{m_1}_{k,\ell}(\arn)$ is
		\[
		(H^{m_0}_{k,\ell}(\mathbb{R}^n),H^{m_1}_{k,\ell}(\mathbb{R}^n))_{\theta}=H^{m_{\theta}}_{k,\ell}(\mathbb{R}^n)\,,\]
		where $m_{\theta}=(1-\theta)m_0+\theta m_1$, $\theta \in (0,1)$.
	\end{enumerate}
\end{thm}
\begin{proof}
	The arguments used to prove the Parts \ref{itm: 1.sb} and \ref{itm:2.sb} are quite standard and will be ommited. For the proof of the complex interpolation in Part \ref{itm: 6.sb}, let as assume that $m_1>m_0$.  For $h \in H^{m_\theta}_{k,l}(\mathbb{R}^n)$, we consider the function
	\[
	f(z):=(A+I)^{\frac{-(zm_1+(1-z)m_0)+m_{\theta}}{2}}h\,.
	\]
	One then can prove that $H^{m_\theta}_{k,l}  \xhookrightarrow{} (H^{m_0}_{k,l},H^{m_1}_{k,l})_{\theta}$, by showing that 
	\[
		\|f(iy)\|_{H^{m_0}}\,,\|f(1+iy)\|_{H^{m_1}} \leq \|h\|_{H^{m_{\theta}}}\,.
	\]
	 On the other hand, by the duality of the complex interpolation we get 
	\[
	(H^{-m_0}_{k,\ell},H^{-m_1}_{k,\ell})_{\theta}=(H^{m_0}_{k,\ell},H^{m_1}_{k,\ell})_{\theta}^{*} \xhookrightarrow{}  (H^{m_\theta}_{k,\ell})^{*}=H^{-m_\theta}_{k,l}\,,
	\]
	and by setting $-t_{\theta}=m_{\theta}$, for $\theta \in [0,1]$, we obtain the reverse inclusion and complete the proof of Part \ref{itm: 6.sb}.
		
		Now, to prove Part \ref{itm:3.sb}, first observe that the dual of $H_{m}^{(b)}$ is $H_{-m}^{(b)}$, and that the spaces $H_{m}^{(b)}$ decrease while $m \in \mathbb{R}$ increases. The following interpolation property needs to be proven
	\begin{equation}\label{inter2}
	(H_{m_0}^{(b)}, H_{m_1}^{(b)})=H_{m_\theta}^{(b)}\,,\quad m_\theta=(1-\theta)m_0+\theta m_1\,,\quad \theta \in (0,1)\,.
	\end{equation}
Assuming that $m_1>m_0$, and letting $h \in H_{m_\theta}^{(b)}$, we consider the function
	\[
	f(z)=e^{z(m_z-m_\theta)}(b^{-m_z+m_\theta})^{w}h\,,\quad \text{where}\quad m_z=(1-z)m_0+zm_1 \,.
	\]
	Indeed, showing that $f(\theta)=h$, and that $	\|f(iy)\|_{H_{m_0}^{(b)}}\,,\|f(1+iy)\|_{H_{m_1}^{(b)}}\leq \|h\|_{H_{m_\theta}^{(b)}}$, we conclude that  $H_{m_\theta}^{(b)}$ is continuously included in the space $(H_{m_0}^{(b)},H_{m_1}^{(b)})_\theta$, while the reverse inclusion can be proven by duality.
	
Let us now prove the identification of the spaces $H^{m}_{k,l}$ and $H_{m}^{(b)}$ in the case where $m \in 2\mathbb{N}$. Indeed we have 
\[
	\|h\|_{H^m} \leq \|(A+I)^{\frac{m}{2}}(b^{-m})^w\|_{\mathcal{L}(L^2)}\|h\|_{H_{m}^{(b)}}\,,\quad h \in H_{m}^{(b)}\,,
\]
and 
	\[
\|h\|_{H_{m}^{(b)}} \leq \|(b^m)^w (A+I)^{-\frac{m}{2}}\|_{\mathcal{L}(L^2)} \|h\|_{H^m}\,,\quad h \in H^{m}_{k,l}\,.
\]

Finally notice that the properties of the calculus, and in particular by Theorem \ref{thm.comp}, the above norms are finite. 

 For the general case where $m \in \ar$, let $m=m_\theta \in \mathbb{R}^{+}$. Then, for some $m_0,m_1 \in 2\mathbb{N}_0$, we have $m_\theta=\theta m_0+(1-\theta)m_1$, and by using the interpolation in \eqref{inter2} 
\[
H^{m_\theta}_{k,\ell}=(H^{m_0}_{k,\ell},H^{m_1}_{k,\ell})_{\theta}=(H_{m_0}^{(b)},H_{m_1}^{(b)})_\theta=H_{m_\theta}^{(b)}\,,
\]
the identification of the spaces $H^{m}_{k,\ell}$ and $H_{m}^{(b)}$ is proven for $m \in \ar^{+}$, while in the case where $m \in \ar^{-}$ it follows by duality. The last shows Part \ref{itm:3.sb} and implies Part \ref{itm:4.sb}.
\end{proof}
\section{Schatten-von Neumann properties}

We shall now obtain some results in relation with the behaviour in Schatten-von Neumann classes for the negative powers of 
 our anharmonic oscillators. The main aim of this section is to provide a simple proof for the main term of the spectral asymptotics of  these operators. In particular our approach simplifies the one given in \cite{hr:sap3, hr:anosc} by Helffer and Robert. In the context of the Weyl-H\"ormander calculus it is useful to recall the following result by Toft  \cite{Toft:Schatten-AGAG-2006}. See also  
 \cite{Buzano-Toft:Schatten-Weyl-JFA-2010} for further developments. 
 
\begin{thm}\label{toft} Let  $g$ be a H\"ormander metric and split, $M$ a $g$-weight and $1\leq r <\infty$. Assume that $h_g^{\frac{k}{2}}M\in L^r(\ardn)$ for some $k\geq 0$, and let $a\in S(M,g)$. Then 
\[a_t(x,D)\in S_r(L^2(\arn)), \,\,\mbox{ for all }t\in\ar.\] 
\end{thm} 
 
\noindent Regarding the classes  introduced here and associated to the anharmonic oscillators we have:
\begin{thm}\label{sch1mf} Let $q=q(\xi)\in \mathcal{P}_{2\ell}$, $p=p(x)\in\mathcal{P}_{2k} $, where $k$ and $\ell$ are integers $\geq 1$. Let $g=g^{(p,q)}$ be as in \eqref{anhmetT} and let $1\leq r<\infty$. Then 
 \[\lambda_g^{-\mu}\in L^r(\ardn),\]
 provided that $\mu>\frac{n}{r}$. \\
Consequently, if $\mu>\frac{n}{r}$ and $a\in S(\lambda_g^{-\mu},g)$, then
 \[a_t(x,D)\in S_r(L^2(\arn)), \,\,\mbox{ for all }t\in\ar.\]
\end{thm}
\begin{proof} In order to estimate the norm $\|\lambda_g^{-\mu}\|_{L^r(\ardn)}$ we see first that since  
\[p_0+q_0+p(x)+q(\xi)\simeq 1+|x|^{2k}+|\xi|^{2\ell},\]
the estimation will be reduced to the latter.

Now for $\gamma>0$ the integral 
 \[\int\limits_{\arn}\int\limits_{\arn}(1+|x|^{2k}+|\xi|^{2\ell})^{-\gamma r}dxd\xi\]
 can be estimated in the subset $B=\{(x,\xi)\in\ardn: \xi_i >0\, , i=1,\dots, n\,\}$. We assume 
 $k\geq \ell,$ the other case being analogous. The change of variable in $B$ given by $(x,\xi)\rightarrow (x_1,\dots,x_n,\xi_1^{\frac{k}{\ell}},\dots,\xi_n^{\frac{k}{\ell}})$ lead us to 
 
\begin{align*} \int\limits_{\arn}\int\limits_{\arn}(1+|x|^{2k}+|\xi|^{2\ell})^{-\gamma r}dxd\xi=&C\int\limits_{B}(1+|x|^{2k}+|\xi|^{2k})^{-\gamma r}\xi_1^{\frac{k}{\ell}-1}\cdots\xi_n^{\frac{k}{\ell}-1}dxd\xi\\
\leq& \,C\int\limits_{B}(1+|x|^{2}+|\xi|^{2})^{-\gamma kr}|\xi|^{(\frac{k}{\ell}-1)n}dxd\xi\\
\leq& \,C\int\limits_{B}(1+|x|^{2}+|\xi|^{2})^{-\gamma kr}|(x,\xi)|^{(\frac{k}{\ell}-1)n}dxd\xi\\
\leq& \,C\int\limits_{B}\langle X\rangle^{(\frac{k}{\ell}-1)n-2\gamma kr}dxd\xi\\
< & \infty\,,
\end{align*}
provided  $(\frac{k}{\ell}-1)n-2\gamma kr<-2n.$ Now, the condition $(\frac{k}{\ell}-1)n-2\gamma kr<-2n$ for the convergence of the integral is reduced to 
$\gamma >\frac{n(k+\ell)}{2k\ell r}$. Now, since $\lambda_g(x,\xi)\simeq (1+|x|^{2k}+|\xi|^{2\ell})^{\frac{k+\ell}{2k\ell}}$, with $\gamma=\mu\frac{k+\ell}{2k\ell }$ in the previous estimate we obtain
\[\|\lambda_g^{-\mu}\|_{L^r(\ardn)}<\infty \, ,\]
 provided $\mu>\frac{n}{r}$.\\
 
Now applying Theorem \ref{toft} with $k=0$, $M=\lambda_g^{-\mu},$ and the the estimate above we have that for all  $a\in S(\lambda_g^{-\mu},g)$, we have
\[a_t(x,D)\in S_r(L^2(\arn)), \,\,\mbox{ for all }t\in\ar.\]
 This concludes the proof of the theorem.
\end{proof}
For the formulation of the following corollary we first recall some facts on the trace class. It is well known that neither the mere integrability of a kernel $K$ on the diagonal  nor the integrability of the symbol on the phase-space are sufficient to guarantee the traceability of the corresponding operator. Under some suitable conditions on the kernel or the symbol one can get such traceability. 

First let us make the following remark regarding the powers of our anharmonic oscillators.
	\begin{rem}\label{rem.h.f}
		One can easily verify that for $q=q(\xi) \in \mathcal{P}_{2\ell}$, $p=p(x) \in \mathcal{P}_{2k}$, the (positive) operator $q(D)+p(x)+p_0+q_0$ arises as the Weyl quantization (and also the $t$-quantization) of its symbol. One can also see that this operator is $g$-elliptic with respect to the metric $g=g^{(p,q)}$ in the sense of  \cite{bufa:hy}. Therefore for any $\mu \in \mathbb{R}$ the operator $(q(D)+p(x)+p_0+q_0)^{-\mu}$ is the Weyl quantization (and also any $t$-quantization) of some symbol in the class $S((q(\xi)+p(x)+p_0+q_0)^{-\mu},g^{(p,q)})$. 
	\end{rem}
	\begin{cor}
		\label{cor.b.1}
		\begin{enumerate}[label=(\alph*)]
\item \label{itm:cor.sch.a} Let $g=g^{(p,q)}$, and let $a \in S(\lambda_{g}^{-\nu},g)$. Then:
			\begin{enumerate}[label=(\roman*)]
				\item If $\nu > \frac{n}{2}$, then for all $t \in \ar$ $a_t(x,D)$, extends to a Hilbert-Schmidt operator on $L^2(\mathbb{R}^n)$, and in particular we have
				\[
				\|a_t(x,D)\|_{S_2}=(2\pi)^{-\frac{n}{2}}\|a\|_{L^2(\mathbb{R}^{2n})}\,.
				\]
				\item If $\nu>n$, then $a_t(x,D)$, $t \in \mathbb{R}$, extends to a trace-class operator on $L^2(\mathbb{R}^n)$, and in particular we have
				\[
				\textnormal{Tr}(a_t(x,D))=(2\pi)^{-n}\int_{\mathbb{R}^n}\int_{\mathbb{R}^n}a(x,\xi)\,dxd\xi\,.
				\]
			\end{enumerate}
\item \label{itm:cor.sch.11} Let  $1\leq r<\infty$,  $q=q(\xi)\in \mathcal{P}_{2\ell},\, p=p(x)\in\mathcal{P}_{2k}$, where $k,\ell$ are integers $\geq 1$.  
If $\mu>\frac{(k+\ell)n}{2k\ell r}$, then 
 \[T^{-\mu}=(q(D)+p(x)+p_0+q_0)^{-\mu}\in S_r(L^2(\arn)).\]
 In particular
 \[((-\Delta)^{\ell}+|x|^{2k}+1)^{-\mu}\in S_r(L^2(\arn)).\]		
		
			\item \label{itm:cor.sch.b}Let $k,\ell \geq 1$ be integers, and let $q=q(\xi) \in \mathcal{P}_{2l}$, $p=p(x) \in \mathcal{P}_{2k}$. For the operator $T^{-\mu}=(q(D)+p(x)+p_0+q_0)^{-\mu}$, where $\mu>0$, the following hold true:
			\begin{enumerate}[label=(\roman*)]
				\item If $\mu>\frac{n(k+l)}{4kl}$, then the operator $T^{-\mu}$ extends to a Hilbert-Schmidt operator on $L^2(\mathbb{R}^n)$ and there exists a constant $C>0$ such that
				\[
				\|T^{-\mu}\|_{S_2}\leq C \|(q(\xi)+p(x)+p_0+q_0)^{-\mu}\|_{L^2(\mathbb{R}^{2n})}\,.
				\] 
				\item If $\mu>\frac{n(k+l)}{2kl}$, then the operator $T^{-\mu}$ extends to a trace-class operator on $L^2(\mathbb{R}^n)$ and there exists a constant $C>0$ such that
				\[
				\textnormal{Tr}(T^{-\mu})\leq C \int_{\mathbb{R}^n}\int_{\mathbb{R}^n}(q(\xi)+p(x)+p_0+q_0)^{-\mu}\,dxd\xi\,.
				\]
			\end{enumerate}
		\end{enumerate}
	\end{cor}
	Before giving the proof, let us recall the following well-known result; see e.g. Nicola and Rodino \cite{NR:book}.
	\begin{rem}\label{NR.rem}
		If $a^w$ extends to a Hilbert-Schmidt operator on $L^2(\mathbb{R}^n)$, then one has
		\[
		\|a^w\|_{S_2}=(2\pi)^{-\frac{n}{2}}\|a\|_{L^2(\mathbb{R}^{2n})}\,,
		\] 
		while if $a^w$ extends to a trace-class operator on $L^2(\mathbb{R}^n)$, then one has 
		\[
		\textnormal{Tr}(a^w)=(2\pi)^{-n}\int_{\mathbb{R}^{2n}}a(x,\xi)dx\,d\xi\,.
		\]
	\end{rem}
	\begin{proof}[Proof of Corollary \ref{cor.b.1}]
The proof of part \ref{itm:cor.sch.a} follows by Theorem \ref{sch1mf} for $r=1,2$ and by Remark \ref{NR.rem}.

For the part \ref{itm:cor.sch.11}, we observe that by Remark  \ref{rem.h.f}
	 the symbol  $\sigma$ of the operator $T^{-\mu}$ belongs to the class  $S(\lambda_g^{-\frac{2k\ell \mu}{k+\ell}},g)$. Now the condition 
\[\frac{2k\ell \mu}{k+\ell}>\frac{n}{r}\]
is equivalent to $\mu> \frac{(k+\ell)n}{2k\ell r}$, and the conclusion now follows from Theorem \ref{sch1mf}. 

The part (c) follows from (a) and observing that for the symbol $\sigma$ of the operator $T^{-\mu}$, then there exists some constant $C^{'}>0$ such that
		\[
		|\sigma(x,\xi)|<C^{'} |(q(\xi)+p(x)+p_0+q_0)^{-\mu}|\,.
		\] Therefore by Remark \ref{NR.rem}, if $\mu>\frac{n(k+\ell)}{4k\ell}$, then for some $C>0$
		\[
		\|T^{-\mu}\|_{S_2}\leq C \|\sigma\|_{L^2(\mathbb{R}^{2n})}\leq C \|(q(\xi)+p(x)+p_0+q_0)^{-\mu}\|_{L^2(\mathbb{R}^{2n})}
		\,,\]
		while if $\mu>\frac{n(k+\ell)}{2k\ell}$, then for some $C>0$
		\begin{eqnarray*}
			\textnormal{Tr}(T^{-\mu}) &= &(2\pi)^{-n} \int_{\mathbb{R}^n}\int_{\mathbb{R}^n} \sigma(x,\xi)\,dx\,d\xi\\
			& \leq & C \int_{\mathbb{R}^n}\int_{\mathbb{R}^n}(q(\xi)+p(x)+p_0+q_0)^{-\mu}\,dxd\xi\,,
		\end{eqnarray*}
		and the proof is complete.
	\end{proof}

We can also collect some consequences for the singular values of the operators we considered. For a compact operator $A$ on a complex Hilbert space $H$ we will denote by $\lambda_j(A)$ the eigenvalues of $A$ ordered in the decreasing order, and $s_j$ the corresponding singular values.
\begin{cor}\label{cor.eigen.asympt.anharm.}
	Let $q:=q(\xi) \in \mathcal{P}_{2\ell}$, $p:=p(x) \in \mathcal{P}_{2k}$, where $k,\ell \geq 1$ integers, and let $1 \leq r < \infty$.
	\begin{enumerate}[label=(\roman*)]
		\item \label{itm:eig.opw}  For $g=g^{(p,q)}$ we have that if $a \in S(\lambda_{g}^{-\nu},g)$, and $\nu>\frac{n}{r}$, then 
		\[
		s_j(a_t(x,D))=o(j^{-\frac{1}{r}})\,,\quad \text{as} \quad j \rightarrow \infty\,,
		\]
		for every $t \in \ar$.
		Consequently, also
		 	\[
		 \lambda_j(a_t(x,D))=o(j^{-\frac{1}{r}})\,,\quad \text{as} \quad j \rightarrow \infty\,,
		 \]
		 for every $t \in \ar$.
		\item \label{itm:eig.anh.pr}For $T^{-\mu}=(q(D)+p(x)+p_0+q_0)^{-\mu}$, if $\mu > \frac{n \left(k+l \right)}{2k\ell r}$, then 
		\begin{equation}\label{sing.asym}
		s_j(T^{-\mu}) =o(j^{-\frac{1}{r}})\,,\quad \text{as}\quad j \rightarrow \infty\,.
		\end{equation}
		Consequently, also
		\begin{equation}\label{eig.asym}
			\lambda_j(T^{-\mu}) =o(j^{-\frac{1}{r}})\,,\quad \text{as}\quad j \rightarrow \infty\,.
		\end{equation}
	\end{enumerate}
\end{cor}

\begin{proof}
	As Corollary \ref{cor.b.1} implies we have $a_t(x,D)\in S_r(\arn)$ for $\nu$ in the range as in the statement. Hence, as it is well known, one can get for the membership of the above class that 
	\[
	s_j(a_t(x,D))=o(j^{-\frac{1}{r}})\,,\quad \text{as}\quad j \rightarrow \infty\,.	\]
	Moreover from the Weyl inequality one has
		\[
	\lambda_j(a_t(x,D))=o(j^{-\frac{1}{r}})\,,\quad \text{as}\quad j \rightarrow \infty\,.	\] This proves \ref{itm:eig.opw}, while similar arguments can be used to prove \ref{itm:eig.anh.pr}.
\end{proof}

 
In the proof above we have applied the {\em Weyl inequality} (cf. \cite{hw:ineq}) which relates the singular values $s_n(T)$ and the eigenvalues $\lambda_n(T)$  for a compact operator $T$ on a complex separable Hilbert space:

\[\sum\limits_{n=1}^{\infty}|\lambda_n(T)|^p\leq \sum\limits_{n=1}^{\infty} s_n(T)^p ,\quad p>0.\]

\begin{rem}
We point out that in the case of the operator $T=(-\Delta)^{\ell}+|x|^{2k}$ and from Theorem 3.2 of \cite{BBR-book} one can obtain
 an estimate for the eigenvalue counting function $N(\lambda)$  of $T$. Indeed, for large $\lambda$ the eigenvalue counting function $N(\lambda)$ is bounded by $C\int_{a(x,\xi)<\lambda} dx d\xi$, where $a(x,\xi)$ is the Weyl symbol of the partial differential operator $T$.
 By the change of variables $\xi=\lambda^{1/2k}\xi'$ and $x=\lambda^{1/2\ell}x'$, we can estimate
for large $\lambda$ that
 \begin{equation}\label{EQ:anhnl}
N(\lambda)\lesssim \iint_{|\xi|^{2k}+|x|^{2\ell}<\lambda} dx d\xi= \lambda^{n(\frac{1}{2k}+\frac{1}{2\ell})}
  \iint_{|\xi'|^{2k}+|x'|^{2\ell}<1} dx' d\xi' \lesssim \lambda^{n(\frac{1}{2k}+\frac{1}{2\ell})}.
\end{equation} 
Now, from \eqref{EQ:anhnl} we can deduce that $(T+I)^{-\mu}\in S_r(L^2(\arn))$ for $\mu>\frac{(k+\ell)n}{2k\ell r}$, 
 and therefore we can recover the same estimate on the decay of singular values and eigenvalues for $(T+I)^{-\mu}$ as we get in \eqref{sing.asym} and \eqref{eig.asym}. The advantage of Corollary \ref{cor.eigen.asympt.anharm.} is that the assumption  $a\in S(\lambda_g^{-\mu},g)$ allows one to consider  more general symbols. 
 
The estimate \eqref{EQ:anhnl} can also be obtained from Theorem 5.5 of \cite{hr:sap3}, 
 where the authors get the following estimate for the eigenvalue counting function of the anharmonic oscillator $T=(-\Delta)^{\ell}+|x|^{2k}$: 
 
 \beq\displaystyle{ N(\lambda)\lesssim \lambda^{n(\frac{1}{2k}+\frac{1}{2\ell})}}\label{hrenn2},\eq
 which again coincides with our estimates.  \\

 Theorem 5.5 of \cite{hr:sap3} gives spectral asymptotics for operators with quasi-homogeneous symbols: $q\sim \sum\limits_{j\geq 0} q_{M-j} $, 
where $M>0$ and where there are integers $k,\ell\geq 1$ such that 
\beq \displaystyle{{q}_{M-j}(\rho^{\ell}x,\rho^k\xi)=\rho^{M-j}q_{M-j}(x,\xi)}\label{condrh}.\eq
Further it is assumed that \\
(i) $Q=q(x,D)$ is formally self-adjoint,\\
(ii) $q(x,\xi)>0$ if $(x,\xi)\in \ardn\setminus\{0\}$.\\

Under these conditions one has the following spectral asymptotics

\beq\displaystyle{ N(\lambda)=\sum\limits_{j=0}^{k+\ell-1} \gamma_j\lambda^{\frac{n(k+\ell)}{M}-\frac{j}{M}}+\obi(\lambda^{\frac{(n-1)(k+\ell)}{M}})}\,\,\mbox{ as }\lambda\rightarrow\infty ,\label{hrenn2b}\eq
for some suitable constants $\gamma_j>0$.\\

Now, by taking $q_M(x,\xi)=|\xi|^{2\ell}+|x|^{2k}$ and $M=2k\ell$ one has
\[q_M(\rho^{\ell}x,\rho^k\xi)=\rho^{2k\ell}|\xi|^{2\ell}+\rho^{2k\ell}|x|^{2k}=\rho^{2k\ell}q_M(x,\xi).\] 
Hence, we can apply \eqref{hrenn2b} for the main term and obtain \eqref{hrenn2}.
\end{rem}
\medskip 
We now derive an immediate consequence on the rate of growth of the eigenvalues of the anharmonic oscillators. 
 First we note that by Corollary \ref{cor.eigen.asympt.anharm.} (ii), for $\mu=1$ and $r>\frac{(k+\ell)n}{2k\ell}$ we get 
 \beq\lambda_j((-\Delta)^{\ell}+|x|^{2k}+1)^{-1}))=\os(j^{-\frac{1}{r}}),\, \mbox{ as }j\rightarrow\infty .\eq
From this  we obtain the following estimate for the rate of growth of the eigenvalues of $(-\Delta)^{\ell}+|x|^{2k}$:\\

For every $L\in\ene$ there exists $L_0\in\ene$ such that
\beq
Lj^{\frac{1}{r}}\leq\lambda_j((-\Delta)^{\ell}+|x|^{2k})),\,\, \mbox{ for }j\geq L_0. 
\eq
Summarising we have obtained the following: 
Let $k, \ell$ be integers $\geq 1$ and $r>\frac{(k+\ell)n}{2k\ell}$. Then,  
for every $L\in\ene$ there exists $L_0\in\ene$ such that
\beq
Lj^{\frac{1}{r}}\leq\lambda_j((-\Delta)^{\ell}+|x|^{2k})),\,\, \mbox{ for }j\geq L_0. 
\eq
Thus, the eigenvalues $\lambda_j((-\Delta)^{\ell}+|x|^{2k}))$ have a growth of  order at least
\beq \label{EQ:growthb}
j^{\frac{1}{r}}, \mbox{ as } j\rightarrow\infty.
\eq

However, we can observe more properties.
From \eqref{hrenn2b} with $M=2k\ell$ for the main term we have

\[N(\lambda)\sim C \lambda^{n(\frac{1}{2k}+\frac{1}{2\ell})} \mbox{ as }\lambda\rightarrow \infty. \]
 
Hence
\[j=N(\lambda_j)\sim C\lambda_j^{n(\frac{1}{2k}+\frac{1}{2\ell})} \mbox{ as } j\rightarrow \infty. \]
Therefore
\[\lambda_j\sim Cj^{\frac{2k\ell}{n(k+\ell)}} \mbox{ as } j\rightarrow \infty, \]
and the  index $\frac{(k+\ell)n}{2k\ell}$ in Corollary \ref{cor.eigen.asympt.anharm.} and in \eqref{EQ:growthb}
can not be improved.




\section{Examples arising from the analysis on Lie groups}\label{SEC:examples}

We now explain how some differential operators arising from the theory of Lie groups fit into the above setting. We start by fixing the notation and recall well known preliminary facts about Lie groups. 
We refer to \cite{FR:book} for more details on the definitions in this part.

Let $G$ be a Lie group and let $\Lie(G)$ be the corresponding Lie algebra. If $\Lie(G)$ can be endowed with a vector space decomposition of the form 
\begin{equation}\label{graded}
\Lie(G)=\bigoplus_{j=1}^{\infty} V_j\,,\quad \text{such that} \quad [V_i,V_j] \subset V_{i+j}\,,
\end{equation}
where all but finitely many $V_j$'s are zero, then we say that $\Lie(G)$ is \textit{graded}. $\Lie(G)$ is called \textit{stratified} if in addition to the condition \eqref{graded}, it can be generated by the first stratum $V_1$, i.e., any element of $\Lie(G)$ can be written as a linear combination of commutators of elements of $V_1$. In this case, if the Lie group $G$ is in addition connected, simply connected, then $G$ is also called graded and stratified, respectively. Finally, if $p$ is the number of the non-zero $V_j$'s in \eqref{graded}, then both $G$ and $\Lie(G)$ are called stratified of step $p$.

Following \cite{ho:hs}, the set of vector fields, say $\{X_1, \cdots,X_r\}$, that spans $V_1$ is said to be a \textit{H\"{o}rmander system}. The operator
\[
\mathcal{L}_{G}=\sum_{j=1}^{r}X_{j}^{2}\,,
\]
for the given H\"{o}rmander system $\{X_1, \cdots,X_r\}$, is the \textit{sub-Laplacian} operator on $G$. Sub-Laplacian operators are a very important tool in non-commutative harmonic analysis, and one can refer to \cite{Fol:sube} or \cite{FR:book}.

Let $\hat{G}$ denote the unitary \textit{dual} of $G$, i.e., the set of equivalence classes of continuous, irreducible unitary representations of $G$ on a Hilbert space $\mathcal{H}$. For a given $\pi \in \hat{G}$, and any $X \in \Lie(G)$, the operator $\pi(X)$ is called the \textit{infinitesimal representation} associated to $\pi$. Each $\pi(X)$ acts on the subspace, say $\mathcal{H}^{\infty}$, of smooth vectors in $\mathcal{H}$. The infinitesimal representation extends to the universal enveloping algebra of the Lie group $G$. For example we are allowed to consider the infinitesimal representation of the sub-Laplacian operator $\mathcal{L}_{G}$. Any $\pi(X)$ is the \textit{global symbol} of $X$ in the sense of the definition in Chapter 5 of \cite{FR:book}.

  In the following we will show that the global symbol of the sub-Laplacian on the Engel and on the Cartan group, nilpotent Lie groups of 3-steps, are examples of the differential operators considered before. The above groups will be treated separately. First a brief description of the structure, and then of the dual of each group is given, leading to the formulas for the global symbol of the sub-Laplacian  operators in both cases. 
 For more pseudo-differential calculus on the Engel and Cartan groups we refer to \cite{chat21}, as well to a more informal presentation in \cite{chat20}. 
 
\subsection{Engel group}

Let $\mathfrak{l}_4=\text{span}\{I_1,I_2,I_3,I_4\}$ be a 3-step stratified Lie algebra, whose generators satisfy the non-zero relations:
\begin{equation}\label{com.rel.engel}
[I_1,I_2]=I_3\,, \quad [I_1,I_3]=I_4\,.
\end{equation}
The corresponding Lie group, called the Engel group, and denoted by $\mathcal{B}_4$, is isomorphic to the manifold $\mathbb{R}^4$. This identification implies that the basis of $\mathfrak{l}_4$ (now called the canonical basis) can be given by the left invariant vector fields 
\begin{equation}\label{left.inv.eng}
\begin{split}
X_1(x)&= \frac{\partial}{\partial x_1}\,, \quad X_2(x)=\frac{\partial}{\partial x_2}-x_1 \frac{\partial}{\partial x_3}+ \frac{x^{2}_{1}}{2}\frac{\partial}{\partial x_4}\,, \\
            X_3(x)&= \frac{\partial}{\partial x_3}-x_1 \frac{\partial}{\partial x_4}\,,\quad X_4(x)= \frac{\partial}{\partial x_4}\,,
\end{split}
\end{equation}
where $x=(x_1,x_2,x_3,x_4) \in \mathbb{R}^{4}$, also satisfying the relations \eqref{com.rel.engel} as expected. The system of vector fields $\{X_1,X_2\}$ is a H\"{o}rmander system and therefore gives rise to the sub-Laplacian (or the canonical sub-Laplacian) on $\mathcal{B}_4$, namely to the operator 
\begin{align*}
\mathcal{L}_{\mathcal{B}_4}&= X_{1}^{2}+X_{2}^{2}\\
&=\frac{\partial^2}{\partial x_{1}^{2}}+ \left(\frac{\partial}{\partial x_2}-x_1 \frac{\partial}{\partial x_3}+ \frac{x^{2}_{1}}{2}\frac{\partial}{\partial x_4} \right)^2\,.
\end{align*}
 Now, given, the description of the dual of the group as Dixmier suggested in \cite{dix:bm}, i.e., a family of operators $$ \hat{B}_{4}=\{\pi_{\lambda, \mu}\arrowvert \lambda\neq 0, \mu \in \mathbb{R}\}$$ acting on $L^2(\mathbb{R})$ via,
\[
\pi_{\lambda, \mu}(x_1,x_2,x_3,x_4)h(u)\equiv \exp \left(i \left(-\frac{\mu}{2\lambda}x_2+\lambda x_4-\lambda x_3 u + \frac{\lambda}{2}x_2 u^2\right)\right)h(u+x_1)\, ,
\]
one can easily find the infinitesimal representation of $\mathfrak{l}_4$ associated to $\pi_{\lambda,\mu}$, and therefore show that the (global)
 symbol of the sub-Laplacian $\mathcal{L}_{\mathcal{B}_4}$ is the family of operators acting on the Schwartz space $\mathcal{S}(\mathbb{R})$, given by 
\begin{equation}\label{symb.sub.eng.}
	\sigma_{\Lh_{\mathcal{B}_4}}(\pi_{\lambda,\mu})= \frac{d^2}{du^2} -\frac{1}{4} \left(\lambda u^2 -\frac{\mu}{\lambda} \right)^2\,,\lambda\in\ar\setminus\{0\},\, \mu\in\ar\,. 
	\end{equation}
 For some fixed parameters $\lambda\in\ar\setminus\{0\},\, \mu\in\ar$, the symbol of an operator as in \eqref{symb.sub.eng.}, is then 
	\[\sigma_{\lambda, \mu}(x,\xi)=\xi^2+\frac{1}{4}\left(\lambda x^2-\frac{\mu}{\lambda}\right)^2=q(\xi)+p_{\lambda,\mu}(x)\,,\] 
where $p\in\mathcal{P}_{4}, q\in\mathcal{P}_2$ and $k=2, \ell=1$, $n=1$.\\ 

We can associate to the operator $\sigma_{\Lh_{B_4}}(\pi_{\lambda,\mu})$ the following  H\"ormander metric on $\ar\times\ar$ corresponding to $p, q$:
\[g_{\lambda, \mu}^{(2,1)}=\frac{dx^2}{(1+\xi^2+\frac{1}{4}(\lambda x^2-\frac{\mu}{\lambda})^2)^{\half}}+\frac{d\xi^2}{1+\xi^2+\frac{1}{4}(\lambda x^2-\frac{\mu}{\lambda})^2}.\]

In terms of the $S(\cdot,g_{\lambda, \mu}^{(2,1)})$ classes, the corresponding constants for the membership to those can be estimated in the form $C_{\alpha\beta}C_{\lambda\mu} $, with $$C_{\lambda\mu}=\max\{1, \lambda^2, \mu, (\frac{\mu}{\lambda})^2\}.$$ 
Therefore, by plugging $C_{\lambda\mu}=\max\{1, \lambda^2, \mu, (\frac{\mu}{\lambda})^2\}$ in the numerators of the metric we get constants  of the form $C_{\alpha\beta}$. See Remark \ref{mtclz}.

On the other hand, in relation with the corresponding Schatten-von Neumann classes,   
by Corollary \ref{cor.b.1} with $k=2, \ell=1$, $n=1$ we have
\beq (I-\sigma_{\Lh_{B_4}}(\pi_{\lambda,\mu}))^{-\gamma}\in S_r(L^2(\ar))\eq
provided $\gamma>\frac{3}{4r}$.

\subsection{Cartan group} 
Let $\mathfrak{l}_5=\text{span}\{I_1,I_2,I_3,I_4,I_5\}$ be a 3-step stratified Lie algebra, whose generators satisfy the non-zero relations:
\begin{equation}\label{com.rel.car.}
[I_1,I_2]=I_3\,, [I_1,I_3]=I_4\,, [I_2,I_3]=I_5\,.
\end{equation}
The corresponding Lie group, called the Cartan group $\mathcal{B}_5$, can be identified with the manifold $\mathbb{R}^5$, while the canonical basis in this case consists of the left invariant vector fields
\begin{equation}\label{left.inv.car}
\begin{split}
X_1(x)&=\frac{\partial}{\partial x_1}\,, \quad X_2(x)=\frac{\partial}{\partial x_2}-x_1 \frac{\partial}{\partial x_3} 
+ \frac{x_1^2}{2} \frac{\partial}{\partial x_4}+x_1x_2\frac{\partial}{\partial x_5}\,, \\
            X_3(x)&= \frac{\partial}{\partial x_3}-x_1 \frac{\partial}{\partial x_4}-x_2 \frac{\partial}{\partial x_5}\,,\quad  X_4(x)= \frac{\partial}{\partial x_4}\,,\quad X_5(x)=\frac{\partial}{\partial x_5},
\end{split}
\end{equation}
where $x=(x_1,x_2,x_3,x_4,x_5) \in \mathbb{R}^5$, also satisfying \eqref{com.rel.car.} as expected. As in the previous case, the left-invariant vector fields $X_1,X_2$, with $X_1,X_2$ are as in \eqref{left.inv.car}, are a H\"{o}rmander system, and the sub-Laplacian (or the canonical sub-Laplacian) on $\mathcal{B}_5$ is given by
\begin{align*}
\mathcal{L}_{\mathcal{B}_5}&= X_{1}^{2}+X_{2}^{2}\\
&=\frac{\partial^2}{\partial x_{1}^{2}}+ \left(\frac{\partial}{\partial x_2}-x_1 \frac{\partial}{\partial x_3} 
+ \frac{x_1^2}{2} \frac{\partial}{\partial x_4}+x_1x_2\frac{\partial}{\partial x_5}\right)^2\,.
\end{align*}
Dixmier in \cite[p.338]{dix:bm} showed that the dual of the group is the family of operators $$\hat{B}_{\lambda, \mu, \nu}=\{\pi_{\lambda, \mu, \nu}\arrowvert \lambda^2+\mu^2\neq 0, \nu \in \mathbb{R}\}$$ acting on $L^2(\mathbb{R})$ via,
\[
\pi_{\lambda, \mu, \nu}(x_1,x_2,x_3,x_4,x_5)h(u)\equiv \exp (i A^{\lambda, \mu, \nu}_{x_1,x_2,x_3,x_4,x_5}(u))h\left( u+\frac{\lambda x_1+\mu x_2}{\lambda^2+\mu^2} \right)\,,
\]
with 
\begin{align*}
A^{\lambda, \mu, \nu}_{x_1,x_2,x_3,x_4,x_5}(u)=& -\frac{1}{2} \frac{\nu}{\lambda^2 +\mu^2}(\mu x_1 -\lambda x_2)+\lambda x_4+ \mu x_5\\
& - \frac{1}{6} \frac{\mu}{\lambda^2 + \mu^2} (\lambda^2 x_1^3+3 \lambda \mu x_1^2 x_2 + 3 \mu^2 x_1 x_2^2 -\lambda \mu x_2^3)\\
& +\mu^2 x_1x_2 u +\lambda \mu (x_1^2-x_2^2) u +\frac{1}{2} (\lambda^2+\mu^2)(\mu x_1-\lambda x_2)u^2\,.
\end{align*}
Thus, one can make use of the infinitesimal representation of  $\mathfrak{l}_5$ deduced, to show that the (global) symbol of the sub-Laplacian
 $\mathcal{L}_{\mathcal{B}_4}$ is the family of operators acting on $\mathcal{S}(\mathbb{R})$, given by
\begin{align}\label{symb.sub.car.}
	\sigma_{\Lh_{\mathcal{B}_5}}(\pi_{\lambda,\mu})=\frac{1}{\lambda^2+ \mu^2} \frac{d^2}{du^2}-\frac{1}{4(\lambda^2 + \mu^2)}((\lambda^2 + \mu^2)^2 u^2 + \nu^2)^2\,.
	\end{align}
	For fixed parameters $\lambda, \mu, \nu$ as above, if we set $\kappa=\lambda^2+\mu^2$, then the symbol of an operator as in \eqref{symb.sub.car.} can be simplified as
\[\sigma_{\kappa,\nu}(x,\xi)=\frac{\xi^2}{\kappa}+\frac{1}{4\kappa}(\kappa^2 x^2+\nu^2)^2=q_{\kappa}(\xi)+p_{\kappa, \nu}(x)\,,\]
with  $p_{\kappa}\in\mathcal{P}_{4}, q_{\kappa, \nu}\in\mathcal{P}_2$, and $k=2, \ell=1$, $n=1$.
Similar properties,
and the same index for Schatten-von Neumann classes are obtained for the negative powers of the quantization of this symbol. Indeed, again 
by Corollary \ref{cor.b.1} with $k=2, \ell=1$, $n=1$ we have
\beq (I-\sigma_{\Lh_{B_5}}(\pi_{\lambda,\mu}))^{-\gamma}\in S_r(L^2(\ar))\eq
provided $\gamma>\frac{3}{4r}$.


\subsection{Heisenberg group}
On the Heisenberg group $\mathbb{H}^n$ we can also consider

\begin{align*}
\pi_{\lambda}\left(\sum\limits_{j=1}^nX_j^{2k}+\sum\limits_{j=1}^nY_j^{2\ell}\right)=&\sum\limits_{j=1}^n\left((\sqrt{|\lambda|}\partial_{u_j})^{2k}+(i\sqrt{|\lambda|}u_j)^{2\ell}\right)\\
=&\sum\limits_{j=1}^n\left(|\lambda|^k\partial_{u_j}^{2k}+(-1)^{\ell}|\lambda|^{\ell}u_j^{2\ell}\right).
\end{align*}

Here we set 

\begin{equation}
 \sqrt{\lambda}=\left\{
\begin{array}{rl}
{\displaystyle\sqrt{\lambda}}\,, & \lambda >0,\\
{\displaystyle -\sqrt{|\lambda|}}\, ,&\lambda <0.
\end{array} \right.
\label{pro2g}\end{equation} 
Hence,
\begin{equation}
 (\sqrt{\lambda})^{2\ell}=\left\{
\begin{array}{rl}
{\displaystyle \lambda^{\ell}}\,, & \lambda >0,\\
{\displaystyle (-1)^{\ell}|\lambda|^{\ell}}\, ,&\lambda <0.
\end{array} \right.
\label{pro2g1}\end{equation} 
Regarding the symbol of the operator above it is given by
\[(-1)^{k}|\lambda|^k\sum\limits_{j=1}^n|\xi_j|^{2k}+(-1)^{\ell}|\lambda|^{\ell}\sum\limits_{j=1}^nu_j^{2\ell}.\]
Here, $\pi_{\lambda}$ are the Schr\"odinger representations of the Heisenberg group. Let $\nu=|\lambda|>0$. Then, the symbol of $\pi_{\lambda}\left(\sum\limits_{j=1}^nX_j^{2k}+\sum\limits_{j=1}^nY_j^{2\ell}\right)$ is given by
\[\sigma_{k,\ell,\nu}(x,\xi)=\nu^{k}\sum\limits_{j=1}^n\xi_j^{2k}+\nu^{\ell}\sum\limits_{j=1}^n x_j^{2\ell}=q_{\nu}(\xi)+p_{\nu}(x),\]
with $p_{\nu}\in\mathcal{P}_{2\ell},\,q_{\nu}\in\mathcal{P}_{2k}$. \\

Therefore, we can associate to the operator $T=\pi_{\lambda}\left(\sum\limits_{j=1}^nX_j^{2k}+\sum\limits_{j=1}^nY_j^{2\ell}\right)$
  a H\"ormander metric $g_{\nu}^{(k,\ell)}$ on $\ar^{2n}$, given by  
\[g_{\nu}^{(k,\ell)}=\frac{C_{\nu}dx^2}{(1+p_{\nu}(x)+q_{\nu}(\xi))^{\frac{1}{k}}}+\frac{C_{\nu}d\xi^2}{(1+p_{\nu}(x)+q_{\nu}(\xi))^{\frac{1}{\ell}}},\]
where $C_{\nu}=\max\{1, \nu^{\max\{k,\ell\}}\}$. 

In terms of the corresponding $S(\cdot, g_{\nu}^{(k,\ell)})$ classes, the constants involved in the membership to these classes
 are of the form $C_{\alpha\beta}$; see Remark \ref{mtclz}.\\

In turn, we can analyse the corresponding Schatten-von Neumann classes for the negative powers of $I-T$. Indeed,    
by Corollary \ref{cor.b.1}, since $p_{\nu}\in\mathcal{P}_{2\ell}$ and $q_{\nu}\in\mathcal{P}_{2k}$ we have
\beq (I-T)^{-\gamma}\in S_r(L^2(\arn)),\eq
provided that $\gamma>\frac{(k+\ell)n}{2k\ell r}$.\\
\medskip

\noindent{\bf Acknowledgments}

The authors were supported by Leverhulme Research Grant RPG-2017-151, 
 FWO Odysseus 1 grant G.0H94.18N: Analysis and Partial Differential Equations, and by the EPSRC grant EP/R003025. The second author was also supported by Vic. Inv Universidad del Valle. Grant No. CI-71234. The authors also want  to express their gratitude to the referee who 
pointed out a number of corrections and valuable suggestions helping to improve the main results and the presentation of the manuscript.




\bibliographystyle{alphaabbr}

\end{document}